\documentclass[12pt,reqno]{amsart}
\usepackage{graphicx, epsfig}
\usepackage{amsfonts,amsmath,amssymb,cite,verbatim}
\usepackage{multirow}
\usepackage{color}
\usepackage{theorem}

 \columnwidth\textwidth
\newcounter{theorem}
\def\openthm#1#2{\refstepcounter{theorem}\bigskip

{\noindent\bf#1~\thetheorem\if#2!{. }\else{ (#2).}\fi}
\it}

\def\thmskip{}
\newenvironment{theorem}[1][!]{\openthm{Theorem}{#1}}{\thmskip}
\newenvironment{lemma}[1][!]{\openthm{Lemma}{#1}}{\thmskip}
\newenvironment{corollary}[1][!]{\openthm{Corollary}{#1}}{\thmskip}
\newenvironment{proposition}[1][!]{\openthm{Proposition}{#1}}{\thmskip}

\newcounter{remark}
\def\openrem#1#2{\refstepcounter{remark}\bigskip
{\noindent \it \bfseries#1~\theremark\if#2!{. }\else{ (#2). }\fi}}
\newenvironment{remark}[1][!]{\openrem{Remark}{#1}}{\qed}

\def\R{\mathbb{R}}
\def\N{\mathbb{N}}

\def\<{\langle}
\def\>{\rangle}

\def\eps{\varepsilon}

\def\E{\mathcal{E}}

\def\Fs{\mathcal{F}}

\def\Us{\mathcal{U}}
\def\As{\mathcal{A}}
\def\Cs{\mathcal{C}}
\def\Fs{\mathcal{F}}

\def\ub{u}
\def\vb{v}

\def\yb{y}

\def\Lat{L^{{\rm a}}_F}
\def\Lqcl{L^{{\rm qcl}}_F}
\def\Lqcf{L^{{\rm qcf}}_F}
\def\Lqnl{L^{{\rm qnl}}_F}

\def\Lqnlconj{\widetilde{L}^{{\rm qnl}}_F}
\def\Lqnlconjf{\widehat{L}^{{\rm qnl}}_F}
\def\Lqceconj{\widetilde{L}^{{\rm qce}}_F}
\def\Lqceconjf{\widehat{L}^{{\rm qce}}_F}
\def\Lqcfconj{\widetilde{L}^{{\rm qcf}}_F}
\def\Lqcfconjf{\widehat{L}^{{\rm qcf}}_F}

\def\Lqce{L^{{\rm qce}}_F}
\def\Li{{L}}

\def\Eqcl{\E^{{\rm qcl}}}
\def\Eqnl{\E^{{\rm qnl}}}
\def\Eqce{\E^{{\rm qce}}}
\def\Ea{\E^{{\rm a}}}

\def\Fcrit{F_*}
\def\Fqcl{\Fs^{{\rm qcl}}}
\def\Fqcf{\Fs^{{\rm qcf}}}
\def\Fat{\Fs^{{\rm a}}}

\def\Fc{\Fs^{{\rm c}}}

\def\Wqcf{{ \widetilde W}}

\def\lqnl{\lambda^{{\rm qnl}}}

\def\yat{y^{{\rm a}}}
\def\yqcl{y^{{\rm qcl}}}
\def\yqcf{y^{{\rm qcf}}}

\def\uat{u^{{\rm a}}}
\def\uqcl{u^{{\rm qcl}}}
\def\uqcf{u^{{\rm qcf}}}
\def\uqnl{u^{{\rm qnl}}}
\def\uqce{u^{{\rm qce}}}

\def\alphaopt{\alpha^{\rm id}_{\rm opt}}
\def\alphamax{\alpha^{\rm id}_{\rm max}}
\def\alphaoptqclonetwo{\alpha^{\rm qcl,1,2}_{\rm opt}}
\def\alphamaxqclonetwo{\alpha^{\rm qcl,1,2}_{\rm max}}

\def\alphaoptqclone{\alpha^{\rm qcl,1,\infty}_{\rm opt}}
\def\alphamaxqclone{\alpha^{\rm qcl,1,{\infty}}_{\rm max}}
\def\alphaoptqcl2{\alpha^{\rm qcl,2,{\infty}}_{\rm opt}}
\def\alphamaxqcl2{\alpha^{\rm qcl,2,{\infty}}_{\rm max}}

\def\Gid{G_{\rm id}(\alpha)}
\def\Gqcl{G_{\rm qcl}(\alpha)}
\def\Gqce{G_{\rm qce}}
\def\Gqclw{\widetilde G_{\rm qcl}(\alpha)}

\def\qnl{{\rm qnl}}
\def\Lamqcf{\Lambda^{\qnl}}
\def\tilS{\tilde{S}}
\def\tilLam{\widetilde{\Lambda}^\qnl}

\DeclareMathOperator{\cond}{cond}

\newcommand{\smfrac}[2]{{\textstyle \frac{#1}{#2}}}
\newcommand{\lpnorm}[2]{\left\|#1\right\|_{\ell^{#2}_\eps}}
\def\half{\smfrac{1}{2}}

\begin{document}

\title[Linear Stationary Iterative Methods for Force-based QC]{Linear Stationary Iterative Methods for the Force-based Quasicontinuum Approximation}

\author{M. Luskin}
\address{School of Mathematics, 206 Church St. SE,
  University of Minnesota, Minneapolis, MN 55455, USA}
\email{luskin@umn.edu}

\author{C. Ortner} \address{Mathematical Institute, St. Giles' 24--29, Oxford OX1 3LB, UK} \email{ortner@maths.ox.ac.uk}

\date{\today}


\thanks{This work was supported in part by the National Science Foundation
under DMS-0757355,
 DMS-0811039, the PIRE Grant OISE-0967140, the Institute for Mathematics and
Its Applications, and
 the University of Minnesota Supercomputing Institute.
This work was also supported by the Department of Energy under
Award Number {DE-SC}0002085. CO was supported by the EPSRC grant
EP/H003096/1 ``Analysis of Atomistic-to-Continuum Coupling
Methods.'' 
}

%%%%%%%%%%%%%%%%%
% Subject class %
%%%%%%%%%%%%%%%%%

\begin{abstract}
  Force-based multiphysics coupling methods have become
  popular since they provide a simple and efficient coupling
  mechanism, avoiding the difficulties in formulating and
  implementing a consistent coupling energy.  They are also the only
  known pointwise consistent methods for coupling a general atomistic
  model to a finite element continuum model.  However, the development
  of efficient and reliable iterative solution methods for the
  force-based approximation presents a challenge due to the
  non-symmetric and indefinite structure of the linearized force-based
  quasicontinuum approximation, as well as to its unusual stability
  properties.  In this paper, we present rigorous numerical analysis
  and computational experiments to systematically study the stability
  and convergence rate for a variety of linear stationary iterative
  methods.
  \end{abstract}

\maketitle
\section{Introduction}
Low energy local minima of crystalline atomistic systems are
characterized by highly localized defects such as vacancies, interstitials,
dislocations, cracks, and grain boundaries separated by large regions where
the atoms are slightly deformed from a lattice structure.  The goal of
atomistic-to-continuum coupling methods
~\cite{LinP:2006a,Gunzburger:2008a,arlequin1,hybrid_review,blan05,gaviniorbital,Ortner:2008a,ted03,Ortiz:1996a,cadd}
is to
approximate a fully atomistic model by maintaining the accuracy of the atomistic model in
small neighbors surrounding the localized defects and using the efficiency of continuum coarse-grained
models in the vast regions that are only mildly deformed from a lattice structure.

Force-based atomistic-to-continuum methods
decompose a computational reference lattice into an
{\it atomistic region} $\mathcal{A}$ and a {\it continuum region}
$\mathcal{C}$, and assign forces to representative atoms
according to the region
they are located in. In the quasicontinuum method, the representative atoms are all atoms in
the atomistic region and the nodes of a finite element approximation
in the continuum region. The force-based approximation is thus given
by~\cite{Dobson:2008a,dobsonluskin08,cadd,doblusort:qcf.stab,qcf.iterative,dobs-qcf2}
\begin{equation*}
\label{qcfdefFa}
\Fqcf_j(\yb):=
\begin{cases}
\Fat_j(\yb)& \text{if $j\in \mathcal{A}$},\\
\Fc_j(\yb)& \text{if $j\in \mathcal{C}$},
\end{cases}
\end{equation*}
where $\yb$ denotes the positions of the representative
atoms which are indexed by $j,$ $\Fat_j(\yb)$ denotes the atomistic
force at representative atom $j,$ and
$\Fc_j(\yb)$ denotes a continuum
force at representative atom $j.$

The force-based quasicontinuum method (QCF) uses a Cauchy-Born strain energy density
for the continuum model to achieve a
patch test consistent approximation~\cite{dobsonluskin08,Miller:2008,dobs-qcf2}.
We recall that a patch test consistent atomistic-to-continuum approximation
exactly reproduces the zero net forces of uniformly strained lattices~\cite{vankoten.blended,Miller:2008,ortner.patch}.
However, the recently
  discovered unusual stability properties of the linearized
  force-based quasicontinuum (QCF) approximation, especially its
  indefiniteness, present a challenge to the development of efficient
  and reliable iterative methods~\cite{qcf.iterative}.
Energy-based quasicontinuum approximations have many attractive
features such as more reliable solution methods, but practical
patch test consistent, energy-based quasicontinuum approximations
have yet to be developed for most
problems of physical interest, such as three-dimensional problems with
many-body interaction potentials~\cite{LuskinXingjie,eam.qc,shapeev}.

Rather than attempt an analysis of linear stationary methods
for the full nonlinear system, in this
paper we restrict our focus to the linearization of a one-dimensional
model problem about the uniform
deformation $y^F$ and consider linear stationary methods of the form
\begin{equation}
  \label{stat.method2}
  P\big(u^{(n+1)}-u^{(n)}  \big)=\alpha r^{(n)},
\end{equation}
where $P$ is a nonsingular preconditioning operator, the damping
parameter $\alpha>0$ is fixed throughout the iteration (that is,
stationary), and the residual is defined as
\begin{displaymath}
r^{(n)}:=f-\Lqcf u^{(n)}.
\end{displaymath}
We will see below that our analysis of this simple model problem
already allows us to observe many interesting and crucial features of
the various methods. For example, we can distinguish which iterative
methods converge up to the critical strain $F_*$ (see~\eqref{critical}
for a discussion of the critical strain), and we obtain first results
on their convergence rates.

We begin in Sections \ref{sec:model} and \ref{sec:qcf} by introducing
the most important quasicontinuum approximations and outlining their
stability properties, which are mostly straightforward generalizations
of results from
\cite{dobs-qcf2,doblusort:qcf.stab,doblusort:qce.stab,qcfspec}.  In
Section \ref{lin.stat}, we review the basic properties of linear
stationary iterative methods.

In Section~\ref{sec:richardson}, we give an
analysis of the Richardson Iteration ($P=I$) and prove
a contraction rate of order $1 - O(N^{-2})$ in the
$\ell^p_\eps$ norm (discrete Sobolev norms are defined in Section~\ref{sec:notation}),
where $N$ is the size of the atomistic system.

In Section~\ref{sec:precond_QCL}, we consider the iterative solution
with preconditioner $P=\Lqcl,$ where $\Lqcl$ is a standard second
order elliptic operator, and show that the preconditioned iteration
with an appropriately chosen damping parameter $\alpha$ is a
contraction up to the critical strain $F_*$ only in $\Us^{2,\infty}$
among the common discrete Sobolev spaces.  We show, however, that a
rate of contraction in $\Us^{2,\infty}$ independent of $N$ can be
achieved with the elliptic preconditioner $\Lqcl$ and an appropriate
choice of the damping parameter $\alpha.$

In Section~\ref{gfc}, we consider the popular ghost force correction
iteration (GFC) which is given by the preconditioner $P=\Lqce,$ and we
show that the GFC iteration ceases to be a contraction for any norm at
strains less than the critical strain. This result and others
presented in Section~\ref{gfc} imply that the GFC iteration might not
always reliably reproduce the stability of the atomistic
system~\cite{doblusort:qce.stab}.  We did not find that the GFC method
predicted an instability at a reduced strain in our benchmark tests
~\cite{bench11} (see also ~\cite{Miller:2008}).  To explain this, we
note that our 1D analysis in this paper can be considered a good model
for cleavage fracture, but not for the slip instabilities studied in
~\cite{bench11, Miller:2008}.  We are currently attempting to develop
a 2D benchmark test for cleavage fracture to study the stability of
the GFC method.

\section{The QC Approximations and Their Stability}
\label{sec:model}
We give a review of the
prototype QC
approximations and their stability properties in this section.
The reader can find more details in
\cite{doblusort:qcf.stab,doblusort:qce.stab}.

\subsection{Function Spaces and Norms}
\label{sec:notation}
We consider a one-dimensional atomistic chain whose $2N+1$ atoms have
the reference positions $x_j = j \eps $ for $\eps = 1/N.$ The displacement
of the boundary atoms will be constrained, so the space of admissible
displacements will be given by the
{\em displacement space}
\begin{displaymath}
  \Us = \big\{ u \in \R^{2N+1} : u_{-N} = u_N = 0 \big\}.
\end{displaymath}
We will use various norms on the space $\Us$ which are discrete
variants of the usual Sobolev norms that arise naturally in the
analysis of elliptic PDEs.

For displacements $v \in \Us$ and $1 \leq p
\leq \infty,$ we define the $\ell^p_\eps$ norms,
\begin{displaymath}
  \lpnorm{v}{p} := \begin{cases}
    \Big( \eps\sum_{\ell=-N+1}^N |v_\ell|^p \Big)^{1/p},
    &1 \leq p < \infty, \\
    \max_{\ell = -N+1,\dots,N} |v_\ell|,
    &p = \infty,
  \end{cases}
\end{displaymath}
and we denote by $\Us^{0,p}$ the space $\Us$ equipped with the
$\ell^p_\eps$ norm. The inner product associated with the
$\ell^2_\eps$ norm is
\begin{equation*}
  \<v,w\> := \eps \sum_{\ell=-N+1}^N v_\ell w_\ell \qquad
  \text{ for } v, w \in \Us.
\end{equation*}
We will also use $\|f\|_{\ell^p_\eps}$ and $\<f, g\>$ to denote the
$\ell^p_\eps$-norm and $\ell^2_\eps$-inner product for arbitrary
vectors $f, g$ which need not belong to $\Us$. In particular, we
further define the $\Us^{1,p}$ norm
\begin{displaymath}
  \|v\|_{\Us^{1,p}} := \|v'\|_{\ell^p_\eps},
\end{displaymath}
where $(v')_\ell = v_\ell'=\eps^{-1}(v_{\ell}-v_{\ell-1})$, $\ell =
-N+1, \dots, N$, and we let $\Us^{1,p}$ denote the space $\Us$
equipped with the $\Us^{1,p}$ norm.  Similarly, we define the space
$\Us^{2,p}$ and its associated $\Us^{2,p}$ norm, based on the centered
second difference $v_\ell'' = \eps^{-2}(v_{\ell+1} - 2 v_\ell +
v_{\ell-1})$ for $\ell=-N+1,\dots,N-1.$

We have that $v'\in\R^{2N}$ for $v\in\Us$ has mean zero
$\sum_{j=-N+1}^{N} v'_j=0.$  We can thus obtain from
\cite[Equation 9]{doblusort:qcf.stab} that
\begin{equation}\label{equiva}
\max_{\substack{ v\in\Us \\ \|v'\|_{\ell^q_\eps = 1} }} \\ \<u',\, v'\>
\le \max_{\substack{\sigma\in\R^{2N}\\ \|\sigma\|_{\ell^q_\eps} = 1}} \<u',\, \sigma\> = \|u\|_{\Us^{1,p}}
\le 2\max_{\substack{ v\in\Us \\ \|v'\|_{\ell^q_\eps = 1} }} \\ \<u',\, v'\>.
\end{equation}

We denote the space of linear functionals on $\Us$ by $\Us^*.$
For $g \in \Us^*$, $s = 0,1,$ and $1 \leq p \leq \infty$, we define
the negative norms $\| g \|_{\Us^{-s,p}}$ by
\begin{equation*}
  \|g\|_{\Us^{-s,p}}
  := \sup_{\substack{v \in \Us\\\|v\|_{\Us^{s,q}} = 1}} \< g, v \>,
\end{equation*}
where $1 \leq q \leq \infty$ satisfies $\frac{1}{p} + \frac{1}{q} =
1.$ We let $\Us^{-s,p}$ denote the dual space $\Us^*$ equipped with
the $\Us^{-s,p}$ norm.

For a linear mapping $A: \Us_1 \to \Us_2$ where $\Us_i$ are vector
spaces equipped with the norms $\|\cdot \|_{\Us_i},$ we denote the
operator norm of $A$ by
\[
\|A\|_{L(\Us_1,\ \Us_2)}:=\sup_{v\in\Us,\,v\ne 0}\frac {\|Av\|_{\Us_2}}{\|v\|_{\Us_1}}.
\]
If $\Us_1=\Us_2$, then we use the more concise notation
\[
\|A\|_{\Us_1} :=\|A\|_{L(\Us_1,\ \Us_1)}.
\]

If $A: \Us^{0,2} \to \Us^{0,2}$ is invertible, then we can define the
{\em condition number} by
\[
\cond(A)=\|A\|_{\Us^{0,2}}\cdot \|A^{-1}\|_{\Us^{0,2}}.
\]
When $A$ is symmetric and positive definite, we have that
\[
\cond(A)=\lambda_{2N-1}^A/\lambda_1^A
\]
where the eigenvalues of $A$ are
$
0<\lambda^A_1\le \dots \le \lambda_{2N-1}^A.
$
If a linear mapping $A: \Us \to \Us$ is symmetric and positive
definite, then we define the $A$-inner product and $A$-norm by
\[
\<v,w\>_A:=\<Av,w\>,\qquad \|v\|_A^2=\<Av,v\>.
\]

The operator $A : \Us_1 \to \Us_2$ is \emph{operator stable} if the operator norm $\|A^{-1}\|_{L(\Us_2,\ \Us_1)}$ is
finite, and a sequence of operators $A_j : \Us_{1,j}
\to \Us_{2,j}$ is \emph{operator stable} if the sequence
$\|(A_j)^{-1}\|_{L(\Us_{2,j},\ \Us_{1,j})}$ is uniformly bounded.  A
symmetric operator $A: \Us^{0,2} \to \Us^{0,2}$ is called \emph{stable} if it is positive
definite, and this implies operator stability.  A sequence of positive
definite, symmetric operators $A_j: \Us^{0,2} \to \Us^{0,2}$ is called {\em {stable}}
if their smallest eigenvalues $\lambda_{1}^{A_j}$ are uniformly bounded away from zero.

\subsection{The atomistic model}
\label{sec:model_at}
We now consider a one-dimensional atomistic chain whose $2N+3$ atoms have
the reference positions $x_j = j\eps $ for $\eps = 1/N,$ and interact
only with their nearest and next-nearest neighbors.

We denote the deformed positions by $y_j$,
$j=-N-1,\dots,N+1;$ and we constrain the boundary atoms and their
next-nearest neighbors to match the uniformly deformed state, $y_j^F =
Fj\eps ,$ where $F>0$ is a macroscopic strain,
that is,
\begin{equation}\label{bc}
\begin{aligned}
  y_{-N-1}&=-F(N+1)\eps,\qquad &y_{-N}&=-F N \eps,\\
y_{N}&=FN\eps,\qquad &y_{N+1}&=F(N+1)\eps.
\end{aligned}
\end{equation}
We introduced the two additional atoms with indices $\pm (N+1)$ so
  that $y = y^F$ is an equilibrium of the
  atomistic model.
The total energy
of a deformation ${\yb} \in \mathbb{R}^{2N+3}$ is now given by
\begin{equation*}
  \Ea(\yb) -\sum_{j=-N}^{N}\eps f_j y_j,
\end{equation*}
where
\begin{equation}\label{atomaa}
\begin{split}
\Ea(\yb)%&:=
\sum_{j=-N}^{N+1} \eps \phi\Big(\frac{y_{j} - y_{j-1}}{\eps} \Big)
%+ \sum_{j = -N+1}^{N+1} \eps \phi\Big(\frac{y_{j} - y_{j-2}}{\eps}\Big)\\
&=\sum_{j=-N}^{N+1} \eps \phi(y_j')
+ \sum_{j = -N+1}^{N+1} \eps \phi(y_j'+y_{j-1}').
\end{split}
\end{equation}
Here, $\phi$ is a scaled two-body interatomic potential (for example,
the normalized Lennard-Jones potential, $\phi(r) = r^{-12}-2 r^{-6}$),
and $f_j$, $j = -N, \dots, N,$ are external forces. The equilibrium equations are given by the force balance
conditions at the unconstrained atoms,
\begin{equation}
\label{full}
\begin{aligned}
-\Fat_j(\yat) &= f_j&\quad&\text{for} \quad j = -N+1, \dots, N-1,\\
\yat_{j}&=Fj\eps&\quad&\text{for} \quad j = -N-1,\,-N,\,N,\,N+1,
\end{aligned}
\end{equation}
where the atomistic force (per lattice spacing $\eps$) is given by
\begin{equation} \label{atomforce}
\begin{split}
\Fat_j(y)
:&=-\frac{1}{\eps}\frac{\partial \E^a(\yb)}{\partial y_j} \\
&=\frac{1}{\eps}\Big\{\left[\phi'(y_{j+1}')
+ \phi'(y_{j+2}'+y_{j+1}')\right]
-\left[\phi'(y_j')
+ \phi'(y_j'+y_{j-1}')\right]\Big\}.
\end{split}
\end{equation}

%\subsubsection{Linearization of $\Fat$.}
%
We linearize \eqref{atomforce} by letting $u \in \R^{2N+3}$, $u_{\pm N} =
u_{\pm (N+1)} = 0$, be a ``small'' displacement from the uniformly
deformed state $y_j^F = Fj\eps;$ that is, we define
\begin{align*}
  & u_j = y_j - y_j^F  \quad \text{ for } j = -N-1,\dots,N+1.
\end{align*}
We then linearize the atomistic equilibrium equations \eqref{full}
about the uniformly deformed state $\yb^F$ and obtain a linear system for the
displacement $\ub^a$,
\begin{equation*}
\begin{aligned}
(\Lat \uat)_j&=f_j&\quad&\text{for} \quad j = -N+1, \dots, N-1,\\
u^{a}_{j}&=0&\quad&\text{for} \quad j = -N-1,\,-N,\,N,\,N+1,
\end{aligned}
\end{equation*}
where $(\Lat\vb)_j$ is
given by
\begin{equation*}
(\Lat \vb)_j:=\phi''_F
      \left[\frac{-v_{j+1} + 2 v_{j} - v_{j-1}}{\eps^2} \right]
      + \phi''_{2F}
      \left[\frac{-v_{j+2} + 2 v_{j} - v_{j-2}}{\eps^2} \right].
\end{equation*}
Here and throughout we define
\[
\phi''_{F} := \phi''(F)\quad\text{and}\quad
\phi''_{2F} := \phi''(2F),
\]
where $\phi$ is the interatomic potential in~\eqref{atomaa}.  We will
always assume that $\phi''_F > 0$ and $\phi''_{2F} < 0,$ which holds
for typical pair potentials such as the Lennard-Jones potential under
physically realistic deformations.

%\subsubsection{Stability of $\Lat$. }
%
The stability properties of $\Lat$ can be understood by using a
representation derived in \cite{doblusort:qce.stab},
\begin{equation}
  \label{eq:Eq_second_diff_form}
  \begin{split}
    \< \Lat u, u \>
    &= \eps A_F \sum_{\ell = -N+1}^N |u_\ell'|^2
    - \eps^3 \phi_{2F}'' \sum_{\ell = -N}^{N} |u_\ell''|^2
    = A_F \|u'\|_{\ell^2_\eps}^2 - \eps^2 \phi_{2F}'' \|u''\|_{\ell^2_\eps}^2, \\
  \end{split}
\end{equation}
where $A_F$ is the {\em continuum elastic modulus}
\[
A_F=\phi_{F}'' + 4 \phi_{2F}''.
\]
We can obtain the following result from the
argument in \cite[Prop. 1]{doblusort:qce.stab} and \cite{qcf.iterative}.

\begin{proposition}
  If $\phi_{2F}'' \leq 0$, then
  \begin{displaymath}
    \min_{ \substack{u \in \R^{2N+3} \setminus\{0\} \\ u_{\pm N} = u_{\pm (N+1)} = 0 } }
    \frac{\< \Lat u, u \>}{\|u'\|_{\ell^2_\eps}^2} = A_F - \eps^2 \nu_{\eps} \phi_{2F}'',
  \end{displaymath}
  where
\[\nu_{\eps}:=\min_{\substack{u \in \R^{2N+3} \setminus \{0\} \\ u_{\pm N} = u_{\pm (N+1)} = 0}}
    \frac{  \|u''\|_{\ell^2_\eps}^2}{ \|u'\|_{\ell^2_\eps}^2 }.
  \]
satisfies $0 < \nu_{\eps} \leq C$ for some universal constant
  $C.$
\end{proposition}

\subsubsection{The critical strain $\Fcrit$} The previous result shows
that $\Lat$ is positive definite, uniformly as $N \to \infty$, if and
only if $A_F > 0$. For realistic interaction potentials, $\Lat$ is
positive definite in a ground state $F_0 > 0$. For simplicity, we
assume that $F_0 = 1$, and we ask how far the system can be
``stretched'' by applying increasing macroscopic strains $F$ until it
loses its stability. In the limit as $N \to \infty$, this happens at
the {\em critical strain} $\Fcrit$, which is the smallest number
larger than $F_0$, solving the equation
\begin{equation}\label{critical}
  A_{\Fcrit} = \phi''(\Fcrit) + 4 \phi''(2 \Fcrit) = 0.
\end{equation}

\subsection{The local QC approximation (QCL)}

The local quasicontinuum (QCL) approximation uses the Cauchy-Born
approximation to approximate the nonlocal atomistic model by a local
continuum model~\cite{Dobson:2008a,Miller:2003a,Ortiz:1996a}.
For next-nearest neighbor interactions, the Cauchy-Born approximation reads
\begin{displaymath}
  \phi\left(\eps^{-1}(y_{\ell+1}-y_{\ell-1})\right) \approx
  \smfrac12 \big[ \phi(2y_\ell') + \phi(2y_{\ell+1}')],
\end{displaymath}
and results in the QCL energy, for $y\in\R^{2N+3}$ satisfying the
boundary conditions~\eqref{bc},
\begin{equation}\label{lqcen}
\begin{split}
\Eqcl(\yb)&= \sum_{j = -N+1}^N \eps \left[\phi(y_j')
+ \phi(2y_j') \right] \\
&\qquad+\eps \left[\phi(y'_{-N})+\frac 12 \phi(2y'_{-N})+\phi(y'_{N+1})
+\frac 12 \phi(2y'_{N+1})\right]%\\
%&= \sum_{j = -N+1}^N \eps \left[\phi(y_j')
%+ \phi(2y_j') \right] +\eps\left[ 2\phi(F)+\phi(2F)\right]
.
\end{split}
\end{equation}
Imposing the artificial boundary conditions of zero displacement from
the uniformly deformed state, $y_j^F = Fj\eps,$ we obtain the QCL
equilibrium equations
\begin{equation*}
\begin{aligned}
-\Fqcl_j(\yqcl) &= f_j&\quad&\text{for} \quad j = -N+1, \dots, N-1,\\
\yqcl_{j}&=Fj\eps& \quad&\text{for} \quad j = -N,\,N,
\end{aligned}
\end{equation*}
where
\begin{equation} \label{lqcforce}
  \begin{split}
    \Fqcl_{j}(\yb):&=-\frac{1}{\eps} \frac{\partial \Eqcl(\yb)}{\partial y_j}\\
    &=\frac{1}{\eps}\Big\{\left[\phi'(y_{j+1}')
      + 2\phi'(2y_{j+1}')\right]
    -\left[\phi'(y_j')
      + 2\phi'(2y_j')\right]\Big\}.
\end{split}
\end{equation}
We see from \eqref{lqcforce} that the QCL equilibrium equations are
well-defined with only a single constraint at each boundary, and we
can restrict our consideration to \linebreak $y\in\R^{2N+1}$ with
$y_{-N}=-F$ and $y_N=F$ as the boundary conditions.

Linearizing the QCL equilibrium equations \eqref{lqcforce} about $y^F$ results in the system
\begin{equation*}
\begin{aligned}
(\Lqcl\uqcl)_j&=f_j&\quad&\text{for} \quad j = -N+1, \dots, N-1,\\
\uqcl_{j}&=0&\quad&\text{for} \quad j = -N,\,N,
\end{aligned}
\end{equation*}
where
\begin{equation*}
\Lqcl=A_F\Li
\end{equation*}
and
$\Li$ is the discrete Laplacian, for $v\in\Us$, given by
\begin{equation}
  \label{eq:defn_Li}
(\Li \vb)_j:=-v_j''=
	\left[\frac{-v_{j+1} + 2 v_{j} - v_{j-1}}{\eps^2} \right], \quad j=-N+1,\dots, N-1.
\end{equation}
The QCL operator is a scaled discrete Laplace operator,
so
\begin{displaymath}
  \< \Lqcl u, u \> = A_F \|u'\|_{\ell^2_\eps}^2 \qquad
  \text{for all } u \in \Us.
\end{displaymath}
In particular, it follows that $\Lqcl$ is stable if and only if $A_F >
0$, that is, if and only if $F < F_*,$ where $F_*$ is the critical
strain defined in~\eqref{critical}.

\subsection{The force-based QC approximation (QCF)}

The force-based quasicontinuum
(QCF) method combines the accuracy of the atomistic model with the
efficiency of the QCL approximation by decomposing the computational reference lattice into an
{\it atomistic region} $\mathcal{A}$ and a {\it continuum region}
$\mathcal{C}$, and assigns forces to atoms according to the region
they are located in. The QCF operator is given
by~\cite{Dobson:2008a,dobsonluskin08}
\begin{equation}
\label{qcfdefF}
\Fqcf_j(\yb):=
\begin{cases}
\Fat_j(\yb)& \text{if $j\in \mathcal{A}$},\\
\Fqcl_j(\yb)& \text{if $j\in \mathcal{C}$},
\end{cases}
\end{equation}
and the QCF equilibrium equations are given by
\begin{equation*}
\begin{aligned}
-\Fqcf_j(\yqcf) &= f_j&\quad&\text{for} \quad j = -N+1, \dots, N-1,\\
\yqcf_{j}&=Fj\eps& \quad&\text{for} \quad j = -N,\,N.
\end{aligned}
\end{equation*}
We note that, since atoms near the boundary belong to
$\mathcal{C}$, only one boundary condition is required at each end.

For simplicity, we specify the atomistic and continuum regions as
follows. We fix $K \in \N$, $1 \leq K \leq N-2$, and define
\begin{displaymath}
  \mathcal{A} = \{-K,\dots,K\} \quad \text{and} \quad
  \mathcal{C} = \{-N+1, \dots, N-1\} \setminus \mathcal{A}.
\end{displaymath}
Linearizing~\eqref{qcfdefF} about $\yb^F$, we obtain
\begin{equation}
  \label{qcf}
  \begin{aligned}
    (\Lqcf \uqcf)_j&=f_j&\quad&\text{for} \quad j = -N+1, \dots, N-1,\\
    \uqcf_{j}&=0&\quad&\text{for} \quad j = -N,\,N,
  \end{aligned}
\end{equation}
where the linearized force-based operator is given explicitly by
\begin{displaymath}
  (\Lqcf v)_j := \left\{
    \begin{array}{ll}
      (\Lqcl v)_j, & \quad\text{for } j \in \Cs, \\
      (\Lat v)_j, & \quad \text{for } j \in \As.
  \end{array} \right.
\end{displaymath}
The stability analysis of the QCF operator $\Lqcf$ is less
straightforward \cite{doblusort:qcf.stab,dobs-qcf2}; we will therefore
treat it separately and postpone it to Section \ref{sec:qcf}.

\subsection{The original energy-based QC approximation (QCE)}
The original energy-based quasicontinuum (QCE) method
\cite{Ortiz:1996a} defines an energy functional by assigning
atomistic energy contributions in the atomistic region and continuum
energy contributions in the continuum region. For our
model problem, we obtain
\begin{displaymath}
  \Eqce(y) = \eps \sum_{\ell \in \As} \E_\ell^a(y)
  + \eps \sum_{\ell \in \Cs} \E_\ell^c(y) \quad \text{for } y \in \R^{2N+1},
\end{displaymath}
where
\begin{align*}
  \E_\ell^c(y) =~& \smfrac12 \big( \phi(2y_\ell') + \phi(y_\ell') + \phi(y_{\ell+1}') + \phi(2y_{\ell+1}') \big), \quad \text{and} \\
  \E_\ell^a(y) =~& \smfrac12 \big( \phi(y_{\ell-1}' + y_\ell') + \phi(y_\ell') + \phi(y_{\ell+1}') + \phi(y_{\ell+1}' + y_{\ell+2}') \big).
\end{align*}

The QCE method is patch tests
inconsistent~\cite{mingyang,Dobson:2008b,Dobson:2008c,Shenoy:1999a},
which can be seen from the existence of ``ghost forces'' at the
interface, that is, $\nabla \Eqce(y^F) = g^F \neq 0$. Hence, the
linearization of the QCE equilibrium equations about $y^F$ takes the
form~(see \cite[Section 2.4]{Dobson:2008b} and \cite[Section
2.4]{Dobson:2008c} for more detail)
\begin{equation}
  \label{qce}
  \begin{aligned}
    (\Lqce\uqce)_j-g_j^F&=f_j&\quad&\text{for} \quad j = -N+1, \dots, N-1,\\
    \uqce_{j}&=0&\quad&\text{for} \quad j = -N,\,N,
  \end{aligned}
\end{equation}
where, for $0 \leq j \leq N-1,$ we have
\begin{equation*}
\begin{split}
(\Lqce \vb)_j &= \phi''_F \frac{-v_{j+1} +2 v_j - v_{j-1}}{\eps^2} \\
& \hspace{-8mm} + \phi_{2F}'' \left\{ \begin{array}{lr}
\displaystyle
 4 \frac{-v_{j+2} +2 v_j - v_{j-2}}{4 \eps^2},
& \hspace{-8mm} 0 \leq j \leq K-2, \\[6pt]
\displaystyle
 4 \frac{-v_{j+2} +2 v_j - v_{j-2}}{4 \eps^2}
+ \frac{1}{\eps} \frac{v_{j+2} - v_{j}}{2 \eps},  & j = K-1, \\[6pt]
\displaystyle
 4 \frac{-v_{j+2} +2 v_j - v_{j-2}}{4 \eps^2}
- \frac{2}{\eps} \frac{v_{j+1} - v_{j}}{\eps}
+ \frac{1}{\eps} \frac{v_{j+2} - v_{j}}{2 \eps},
& j = K, \\[6pt]
\displaystyle
4 \frac{-v_{j+1} +2 v_j - v_{j-1}}{\eps^2}
-  \frac{2}{\eps} \frac{v_{j} - v_{j-1}}{\eps}
+ \frac{1}{\eps} \frac{v_{j} - v_{j-2}}{2 \eps}, & \quad j = K+1, \\[6pt]
\displaystyle
4 \frac{-v_{j+1} +2 v_j - v_{j-1}}{\eps^2}
+ \frac{1}{\eps} \frac{v_{j} - v_{j-2}}{2 \eps}, & j = K+2, \\[6pt]
\displaystyle
4 \frac{-v_{j+1} +2 v_j - v_{j-1}}{\eps^2}, &
\hspace{-16mm} K+3 \leq j \leq N-1,
\end{array}\right.
\end{split}
\end{equation*}
and where the vector of ``ghost forces,'' $g$, is defined by
\begin{equation*}
g_j^F = \begin{cases}
0, & 0 \leq j \leq K-2, \\
-\frac{1}{2\eps} \phi'_{2F}, & j = K-1, \\
\hphantom{-} \frac{1}{2\eps} \phi'_{2F}, & j = K, \\
\hphantom{-} \frac{1}{2\eps} \phi'_{2F}, & j = K+1, \\
-\frac{1}{2\eps} \phi'_{2F}, & j = K+2, \\
0, &  K+3 \leq j \leq N-1.\\
\end{cases}
\end{equation*}
The equations
for $j=-N+1,\dots,-1$ follow from symmetry.

The following result is a new sharp stability estimate for the QCE
operator $\Lqce$. Its somewhat technical proof is given in Appendix
\ref{sec:app_gfc}.

\begin{theorem}
  \label{th:stab_qce}
  If $K \geq 1$, $N \geq K+2$, and $\phi_{2F}'' \leq 0$, then
  \begin{displaymath}
    \inf_{\substack{u \in \Us \\ \|u'\|_{\ell^2_\eps} = 1}} \< \Lqce u, u\>
    = A_F + \lambda_K \phi_{2F}'',
  \end{displaymath}
  where $\smfrac12 \leq \lambda_K \leq 1$. Asymptotically, as $K \to
  \infty$, we have
  \begin{displaymath}
    \lambda_K \sim \lambda_* + O(e^{-cK}) \quad
    \text{where } \lambda_* \approx  0.6595
    \text{ and } c \approx 1.5826.
  \end{displaymath}
\end{theorem}

\subsection{The quasi-nonlocal QC approximation (QNL)}
\label{sec:model:qnl}
The QCF method is the simplest idea to circumvent the interface
inconsistency of the QCE method, but gives non-conservative
equilibrium equations~\cite{Dobson:2008a}. An alternative
energy-based approach was suggested
in \cite{Shimokawa:2004,E:2006}, which is based on a modification of
the energy at the interface. The quasi-nonlocal approximation (QNL) is given by the energy functional
\begin{equation*}
  \begin{split}
    \Eqnl(y) := \eps \sum_{\ell = -N+1}^N  \phi(y_\ell')
    + \eps \sum_{\ell \in \As} \phi(y_\ell' + y_{\ell+1}')
    + \eps \sum_{\ell \in \Cs} \smfrac12 \big[ \phi(2y_\ell')
    + \phi(2y_{\ell+1}')\big],
  \end{split}
\end{equation*}
where we set $\phi(y_{-N}')=\phi(y_{N+1}')=0$. The QNL approximation
is patch test consistent; that is, $y = y^F$ is an equilibrium of the QNL energy
functional.

The linearization of the QNL equilibrium equations about $y^F$ is
\begin{equation*}
  \begin{aligned}
    (\Lqnl \uqnl)_j&=f_j&\quad&\text{for} \quad j = -N+1, \dots, N-1,\\
  \uqnl_{j}&=0&\quad&\text{for} \quad j = -N,\,N,
  \end{aligned}
\end{equation*}
where
\begin{equation}
\label{Lqnl}
\begin{split}
(\Lqnl v)_j &= \phi''_F \frac{-v_{j+1} +2 v_j - v_{j-1}}{\eps^2} \\
& \hspace{-5mm} + \phi_{2F}'' \left\{ \begin{array}{lr}
\displaystyle
 4 \frac{-v_{j+2} +2 v_j - v_{j-2}}{4 \eps^2},
& 0 \leq j \leq K-1, \\[6pt]
\displaystyle
 4 \frac{-v_{j+2} +2 v_j - v_{j-2}}{4 \eps^2}
- \frac{-v_{j+2} + 2 v_{j+1} - v_{j}}{\eps^2},
& j = K, \\[6pt]
\displaystyle
4 \frac{-v_{j+1} +2 v_j - v_{j-1}}{\eps^2}
+ \frac{-v_{j} + 2 v_{j-1} - v_{j-2}}{\eps^2},
& j = K+1, \\[6pt]
\displaystyle
4 \frac{-v_{j+1} +2 v_j - v_{j-1}}{\eps^2}, &
\hspace{-5mm} K+2 \leq j \leq N-1.
\end{array} \right.
\end{split}
\end{equation}

We can repeat our stability analysis for the periodic QNL operator in
\cite[Sec. 3.3]{doblusort:qce.stab} verbatim to obtain the following
result.

\begin{proposition}
  \label{th:stab_qnl}
  If $K < N-1$, and $\phi_{2F} \leq 0$, then
  \begin{displaymath}
    \inf_{\substack{u \in \Us \\ \|u'\|_{\ell^2_\eps} = 1}} \< \Lqnl u, u \> = A_F.
  \end{displaymath}
\end{proposition}

\begin{remark}\label{depend}
  Since $\phi_{2F}''=(A_F-\phi_F'')/4,$ the linearized operators
  $(\phi''_F)^{-1}\Lat,$ $(\phi''_F)^{-1}\Lqcl,$
  $(\phi''_F)^{-1}\Lqcf,$ $(\phi''_F)^{-1}\Lqce,$ and
  $(\phi''_F)^{-1}\Lqnl$ depend only on $A_F/\phi''_F$, $N$ and $K$.
\end{remark}

\section{Stability and Spectrum of the QCF operator}
\label{sec:qcf}
In this section, we give various properties of the linearized QCF
operator, most of which are variants of our results in
\cite{dobs-qcf2,doblusort:qcf.stab}. We first give a result
for the non-coercivity of the QCF operator
which lies at the heart of many of the difficulties one encounters in
analyzing the QCF method.

\begin{theorem}[Theorem 1, \cite{dobs-qcf2}]
  \label{th:nocoerc}
  If $\phi_F'' > 0$ and $\phi_{2F}'' \in \mathbb{R} \setminus \{0\}$
  then, for sufficiently large $N,$ the operator $\Lqcf$ is {\em not}
  positive-definite. More precisely, there exist $N_0 \in \mathbb{N}$
  and $C_1 \geq C_2 > 0$ such that, for all $N \geq N_0$ and $2 \leq K
  \leq N/2$,
  \begin{displaymath}
    - C_1 N^{1/2} \leq  \inf_{\substack{\vb \in \Us \\ \|\vb'\|_{\ell^2_\eps} = 1}}
   \big \langle \Lqcf {\vb}, {\vb} \big\rangle
    \leq - C_2 N^{1/2}.
  \end{displaymath}
\end{theorem}

The proof of Theorem \ref{th:nocoerc} yields also
the following asymptotic result on the operator norm of $\Lqcf$.  Its
proof is a straightforward extension of \cite[Lemma 2]{dobs-qcf2},
which covers the case $p = 2$, and we therefore omit it.

\begin{lemma}\label{lpbound}
  Let $\phi_{2F}'' \neq 0$, then there exists a constant $C_3 > 0$ such
  that for sufficiently large $N$, and for $2 \leq K \leq N/2$,
  \begin{displaymath}
    C_3^{-1} N^{1/p} \leq \big\| \Lqcf \big\|_{L(\Us^{1,p},\ \Us^{-1,p})}
    \leq C_3 N^{1/p}.
  \end{displaymath}
\end{lemma}

As a consequence of Theorem \ref{th:nocoerc} and Lemma~\ref{lpbound}, we analyzed the
stability of $\Lqcf$ in alternative norms. By following the proof of
\cite[Theorem 3]{doblusort:qcf.stab} verbatim (see also \cite[Remark
3]{doblusort:qcf.stab}), we can obtain the following sharp stability result.

\begin{proposition}
  \label{th:sharp_stab_U2inf}
  If $A_F > 0$ and $\phi_{2F}'' \leq 0$, then $\Lqcf$ is invertible with
  \begin{displaymath}
    \big\|(\Lqcf)^{-1}\big\|_{L(\Us^{0,\infty},\ \Us^{2,\infty})} \leq 1/A_F.
  \end{displaymath}
  If $A_F = 0,$ then $\Lqcf$ is singular.
\end{proposition}

\medskip This result shows that $\Lqcf$ is operator stable up to the
critical strain $\Fcrit$ at which the atomistic model loses its
stability as well (cf. Section \ref{sec:model_at}).

\subsection{Spectral properties of $\Lqcf$ in $\Us^{0,2}=\ell^2_\eps$}
\label{sec:qcf:U02}
The spectral properties of the $\Lqcf$ operator are fundamental for
the analysis of the performance of iterative methods in Hilbert
spaces.  The basis of our analysis of $\Lqcf$ in the Hilbert space
$\Us^{0,2}$ is the surprising observation that, even though $\Lqcf$ is
non-normal, it is nevertheless diagonalizable and its spectrum is
identical to that of $\Lqnl$. We first observed this numerically in
\cite[Section~4.4]{doblusort:qcf.stab} for the case of periodic
boundary conditions. A proof has since been given in \cite[Section
3]{qcfspec}, which translates verbatim to the case of Dirichlet
boundary conditions and yields the following result.

\begin{lemma}
  \label{th:samespec}
  For all $N \geq 4,\ 1 \leq K \leq N-2$, we have the identity
  \begin{equation}
    \label{eq:sim_LqcfLqnl}
    \Lqcf = \Li^{-1} \Lqnl \Li.
  \end{equation}
  In particular, the operator $\Lqcf$ is diagonalizable and its
  spectrum is identical to the spectrum of $\Lqnl$.
\end{lemma}

\medskip
We denote the eigenvalues of $\Lqnl$ (and $\Lqcf$) by
\begin{displaymath}
  0<\lqnl_1\le ... \lqnl_{\ell}\le ...\le \lqnl_{2N-1}.
\end{displaymath}
The following lemma gives a lower bound for $\lqnl_{1},$ an upper
bound for $\lqnl_{2N-1},$ and consequently an upper bound for
$\cond(\Lqnl) = {\lqnl_{2N-1}}/{\lqnl_{1}}$.

% Assuming the validity of
% Lemma \ref{th:samespec} (which has now been proved in \cite{qcfspec}) we proved the following result on
% the spectrum of $\Lqcf$ in ~\cite{qcf.iterative}.

\begin{lemma}\label{qnlcond}
If $K < N-1$ and $\phi_{2F}'' \leq 0$, then
\begin{equation*}
\begin{gathered}
\lqnl_{1}\ge  2\,A_F, \qquad
\lqnl_{2N-1}\le \left(A_F-4\phi_{2F}''\right)\eps^{-2}=\phi_{F}''\eps^{-2},
\quad \text{and} \\
\cond(\Lqnl) =
\frac{\lqnl_{2N-1}}{\lqnl_{1}} \le
\left(\frac{\phi_{F}''}{2A_F}\right)\eps^{-2}.
\end{gathered}
\end{equation*}
\end{lemma}

\medskip For the analysis of iterative methods, we are also interested
in the condition number of a basis of eigenvectors of $\Lqcf$ as $N$
tends to infinity.
% \alertco{Assuming the
% validity of Conjecture
Employing Lemma \ref{th:samespec}, we can write $\Lqcf = \Li^{-1}
\Lambda^{{\rm qcf}} \Li$ where $\Li$ is the discrete Laplacian
operator and $\Lambda^{{\rm qcf}}$ is diagonal. The columns of
$\Li^{-1}$ are poorly scaled; however, a simple rescaling was found in
\cite[Thm. 3.3]{qcfspec} for periodic boundary conditions. The
construction and proof translate again verbatim to the case of
Dirichlet boundary conditions and yield the following result (note, in
particular, that the main technical step, \cite[Lemma 4.6]{qcfspec}
can be applied directly).
%  We gave numerical results in
% \cite{qcf.iterative} for the Dirichlet problem that support the
% $o\left(\log(N)\right)$ growth of ${\cond}(V).$ A stronger result
% ${\cond}(V)=O(1)$ for the periodic problem has been recently proven in
% ~\cite{qcfspec}.
\begin{lemma}
  \label{th:qcf:cond_U02}
  Let $A_F > 0$, then there exists a matrix $V$ of eigenvectors for
  the force-based QC operator $\Lqcf$ such that $\cond(V)$ is bounded
  above by a constant that is independent of $N$.
\end{lemma}

\subsection{Spectral properties of $\Lqcf$ in $\Us^{1,2}$}
\label{sec:qcf:U12}
In our analysis below, particularly in Sections \ref{sec:P_QCL_U2inf}
and \ref{sec:P_QCF_U1inf}, we will see that the preconditioner $\Lqcl=A_F\Li$
is a promising candidate for the efficient solution of the QCF system.
The operator $\Li^{1/2}$ can be understood as a basis
transformation to an orthonormal basis in $\Us^{1,2}$. Hence, it will
be useful to study the spectral properties of $\Lqcf$ in that
space. The relevant (generalized) eigenvalue problem is
\begin{equation}
  \label{eq:defn_U12-evals}
  \Lqcf v = \lambda \Li v,\qquad v\in\Us,
\end{equation}
which can, equivalently, be written as
\begin{equation}\label{eigv}
  \Li^{-1} \Lqcf v = \lambda v, \qquad v \in \Us,
\end{equation}
or as
\begin{equation}\label{eigw}
  \Li^{-1/2} \Lqcf \Li^{-1/2} w = \lambda w,\qquad w\in\Us,
\end{equation}
with the basis transform $w = \Li^{1/2} v$, in either case reducing it
to a standard eigenvalue problem in $\ell^2_\eps$.
%
% We presented numerical experiments in ~\cite{qcf.iterative}
% that are consistent with the
% $\Us^{1,2}$-spectra of the $\Lqcf$ and $\Lqnl$ operators being identical
% to numerical precision and give this result as a conjecture.
% For periodic boundary conditions, this has been proven
% in ~\cite{qcfspec}.
Since $\Li$ and $\Li^{1/2}$ commute, Lemma \ref{th:samespec}
immediately yields the following result.
\begin{lemma}
  \label{th:spec_U12}
  For all $N \geq 4,\ 1 \leq K \leq N-2$ the operator $\Li^{-1}\Lqcf$
  is diagonalizable and its spectrum is identical to the spectrum of
  $\Li^{-1}\Lqnl$.
\end{lemma}

We gave a proof in \cite{qcf.iterative} of the following lemma, which
completely characterizes the spectrum of $\Li^{-1} \Lqnl$, and thereby
also the spectrum of $\Li^{-1}\Lqcf$. We denote the spectrum of
$\Li^{-1}\Lqnl$ (and $\Li^{-1} \Lqcf$) by $\{ \mu_j^{\rm qnl} : j = 1,
\dots, 2N-1\}$.

\begin{lemma}
  \label{th:spec_Lqnl_U12}
  Let $K \leq N-2$ and $A_F > 0$, then the (unordered) spectrum of
  $\Li^{-1}\Lqnl$ (that is, the $\Us^{1,2}$-spectrum) is given by
  \begin{align*}
    \mu_j^{\rm qnl} =
    \begin{cases}
      A_F - {4 \phi_{2F}''} \sin^2\big(\smfrac{j \pi}{4 K+4}\big), &
      j = 1, \dots, 2K+1, \\
      A_F, & j = 2K+2, \dots, 2N-1.
    \end{cases}
  \end{align*}
  In particular, if $\phi_{2F}'' \leq 0,$ then
  \begin{displaymath}
    \frac{\max_j \mu_j^{\rm qnl}}{\min_j \mu_j^{\rm qnl}}
    = 1 - \frac{4 \phi_{2F}''}{A_F}
      \sin^2\left(\smfrac{(2K+1)\pi}{4K+4}\right)
      =\frac{\phi_F''}{A_F}+\frac{4 \phi_{2F}''}{A_F}
      \sin^2\left(\smfrac{\pi}{4K+4}\right)
   = \frac{\phi_F''}{A_F} + O(K^{-2}).
  \end{displaymath}
\end{lemma}

We conclude this study by stating a result on the condition number of
the matrix of eigenvectors for the eigenvalue
problem~\eqref{eigw}. Letting $\tilde{V}$ be an orthogonal matrix of
eigenvectors of $\Li^{-1/2} \Lqnl \Li^{-1/2}$ and $\tilde\Lambda$ the
corresponding diagonal matrix, then Lemma \ref{th:samespec} yields
\begin{align*}
  \Li^{-1/2} \Lqcf \Li^{-1/2} =~& \Li^{-1} \big[ \Li^{-1/2} \Lqnl
  \Li^{-1/2} \big] \Li  \\
  =~&  (\tilde{V}^T \Li)^{-1} \tilde{\Lambda} (\tilde{V}^T \Li).
\end{align*}
Clearly, $\cond(\tilde{V}^T \Li) = O(N^2)$, which gives the following result.

\begin{lemma}
  \label{th:qcf:condP_U02w}
  If $A_F > 0,$ then there exists a matrix $\Wqcf$ of eigenvectors for
  the preconditioned force-based QC operator $\Li^{-1/2} \Lqcf
  \Li^{-1/2}$, such that ${\cond}(\Wqcf) = O(N^2)$ as $N \to \infty$.
\end{lemma}

\section{Linear Stationary Iterative Methods}
\label{lin.stat}
In this section, we investigate linear stationary iterative methods to
solve the linearized QCF equations~\eqref{qcf}. These are iterations
of the form
\begin{equation}\label{stat.method}
  P\big(u^{(n)}-u^{(n-1)}  \big)=\alpha r^{(n-1)},
\end{equation}
where $P$ is a nonsingular preconditioner, the step size parameter
$\alpha>0$ is constant (that is, stationary), and the residual is
defined as
\begin{displaymath}
r^{(n)}:=f-\Lqcf u^{(n)}.
\end{displaymath}

The iteration error
\begin{displaymath}
e ^{(n)}:=\ub^{\rm qcf}-u^{(n)}
\end{displaymath}
satisfies the recursion
\begin{displaymath}
Pe ^{(n)}=\big(P-\alpha \Lqcf\big)e ^{(n-1)},
\end{displaymath}
or equivalently,
\begin{equation}\label{rec}
  e ^{(n)}=\big(I-\alpha P^{-1}\Lqcf\big) e ^{(n-1)}=: Ge ^{(n-1)},
\end{equation}
where the operator $G=I-\alpha P^{-1}\Lqcf:\Us\to\Us$ is called the
{\em iteration matrix}.  By iterating~\eqref{rec}, we obtain that
\begin{equation}\label{Gn}
e ^{(n)}=\big(I-\alpha P^{-1}\Lqcf\big)^n e ^{(0)}=G^ne ^{(0)}.
\end{equation}

Before we investigate various preconditioners, we briefly review the
classical theory of linear stationary iterative methods~\cite{saad}.
We see from~\eqref{Gn} that the iterative method \eqref{stat.method}
converges for every initial guess $u^{(0)}\in\Us$ if and only if $G^n
\to 0$ as $n\to\infty.$ For a given norm $\|v\|$, for $v\in\Us,$ we
can see from \eqref{Gn} that the reduction in the error after $n$
iterations is bounded above by
\begin{equation*}
\|G^n\|=\sup_{e^{(0)}\in\Us}\frac{\|e^{(n)}\|}{\|e^{(0)}\|}.
\end{equation*}

It can be shown~\cite{saad} that the convergence of the iteration for
every initial guess $u^{(0)}\in\Us$ is equivalent to the condition
$\rho(G)<1,$ where $\rho(G)$ is the {\em spectral radius} of $G$,
\begin{displaymath}
\rho(G) = \max\left\{|\lambda_i| : \lambda_i \text{ is
an eigenvalue of } G\right\}.
\end{displaymath}
In fact, the Spectral Radius Theorem~\cite{saad} states that
\begin{equation*}
\lim_{n\to\infty}\|G^n\|^{1/n}=\rho(G)
\end{equation*}
for any vector norm on $\Us.$ However, if $\rho(G) < 1$ and $\|G\|\ge
1,$ the Spectral Radius Theorem does not give any information about
how large $n$ must be to obtain \mbox{$\|G^n\|\le 1.$} On the other hand, if
$\rho(G) < 1,$ then there exists a norm $\|\cdot\|$ such that
$\|G\|<1$, so that $G$ itself is a contraction~\cite{keller}. In this
case, we have the stronger contraction property that
\begin{displaymath}
  \|e^{(n)}\|\le \|G\| \|e ^{(n-1)}\|\le \|G\|^n \|e ^{(0)}\|.
\end{displaymath}

In the remainder of this section, we will analyze the norm of the
iteration matrix, $\|G\|,$ for several preconditioners $P,$ using
appropriate norms in each case.
\begin{comment}
We recall that the
{\em asymptotic rate of convergence} $\mathcal{R}(G)$ of a method with
iteration matrix $G$ is defined by
\begin{equation}\label{rate}
\mathcal{R}(G)=-\ln\rho(G).
\end{equation}
\end{comment}

\section{The Richardson Iteration ($P=I$)}
\label{sec:richardson}
The simplest example of a linear iterative method is the Richardson
iteration, where $P = I$. If follows from Lemma \ref{th:qcf:cond_U02}
that there exists a similarity transform $S$ such that
\begin{equation}
  \label{simil}
  \Lqcf = S^{-1} \Lamqcf S,
\end{equation}
where $\cond(S) \leq C$ (where $C$ is independent of $N$), and
$\Lamqcf$ is the diagonal matrix of $\Us^{0,2}$-eigenvalues
$(\lqnl_j)_{j = 1}^{2N-1}$ of $\Lqcf$. As an immediate consequence, we
obtain the identity
\begin{displaymath}
  \Gid = I-\alpha \Lqcf = S^{-1} \big(I-\alpha \Lamqcf \big)S,
\end{displaymath}
where yields
\begin{equation}
  \label{est2}
  \| \Gid \|_{\ell^2_\eps} \leq \cond(S) \| I - \alpha \Lamqcf
  \|_{\ell^2_\eps}
  \leq C \max_{j = 1, \dots, 2N-1} \big| 1 - \alpha \lqnl_j \big|.
\end{equation}
If $A_F > 0$, then it follows from Proposition \ref{th:stab_qnl} that
$\lqnl_j > 0$ for all $j$, and hence that the iteration matrix
$\Gid:=I-\alpha \Lqcf$ is a contraction in the
$\|\cdot\|_{\ell^2_\eps}$ norm if and only if
$0<\alpha<\alphamax:=2/\lqnl_{2N-1}.$ It follows from
Lemma~\ref{qnlcond} that $\alphamax\le (2\eps^2)/\phi_F''.$

We can minimize the contraction constant for $\Gid$ in the $\|v\|_{S^T S}$
norm by choosing
$
\alpha = \alphaopt:=2/({\lqnl_{1}+\lqnl_{2N-1}}),
$
and in this case we obtain from Lemma~\ref{qnlcond} that
\begin{equation*}
  \big\|G_{\rm id}\big(\alphaopt\big)\big\|_{\ell^2_\eps}
  \leq C \frac{\lqnl_{2N-1}-\lqnl_{1}}{\lqnl_{2N-1}+\lqnl_{1}}
  \le C \bigg( 1-\frac{2A_F\eps^2} {\phi_{F}''} \bigg).
\end{equation*}
It thus follows that the contraction constant for $\Gid$ in the
$\|\cdot\|_{\ell^2_\eps}$ norm is only of the order
$1-\text{O}(\eps^2),$ even with an optimal choice of $\alpha.$ This is
the same generic behavior that is typically observed for Richardson
iterations for discretized second-order elliptic differential
operators.

\subsection{Numerical example for the Richardson Iteration}In Figure~\ref{fig:stat_nop}, we plot the error in the Richardson iteration
against the iteration number. As a typical example, we use the right-hand side
\begin{equation}
  \label{eq:num_ex_rhs}
  f(x) = h(x) \cos(3 \pi x) \quad \text{where} \quad
  h(x) = \begin{cases}
    \ \ 1, & x \geq 0, \\
    -1, & x < 0,
  \end{cases}
\end{equation}
which is smooth in the continuum region but has a discontinuity in the
atomistic region. We choose $\phi_F''=1,$ $A_F=0.5,$ and the optimal $\alpha=\alphaopt$
discussed above
(we note that $G_{\rm id}(\alphaopt)$ depends only on
$A_F/\phi_F''$ and $N,$ but $e^{(0)}$ depends on $A_F$ and
$\phi_F''$ independently)
. We observe initially a much faster convergence rate
than the one predicted because the initial residual for ~\eqref{eq:num_ex_rhs}
has a large component in the eigenspaces corresponding to the intermediate
eigenvalues $\lqnl_{j}$ for $1<j<2N-1.$  However, after a few iterations the
convergence behavior approximates the predicted rate.

\begin{figure}
  \begin{center}
    \includegraphics[width=9cm]{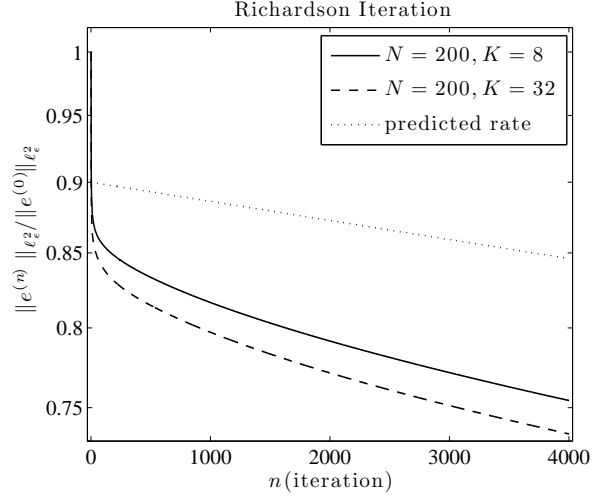}
  \end{center}
  \caption{\label{fig:stat_nop} Normalized $\ell^2_\eps$-error of
    successive Richardson iterations for the linear QCF system with $N
    = 200, K = 8,\, 32$, $\phi_F''=1,$ $A_F=0.5,$  optimal
    $\alpha=\alphaopt,$ right-hand side
    \eqref{eq:num_ex_rhs}, and starting guess $u^{(0)} = 0.$
  }
\end{figure}

\section{Preconditioning with QCL ($P=\Lqcl=A_F \Li$)}
\label{sec:precond_QCL}
We have seen in Section \ref{sec:richardson} that the Richardson
iteration with the trivial preconditioner $P = I$ converges slowly,
and with a contraction rate of the order $1 - O(\eps^2)$. The goal of
a (quasi-)optimal preconditioner for large systems is to obtain a
performance that is independent of the system size. We will show in
the present section that the preconditioner $P = A_F \Li$ (the system
matrix for the QCL method) has this desirable quality.

Of course, preconditioning with $P = A_F \Li$ comes at the cost of
solving a large linear system at each iteration. However, the QCL
operator is a standard elliptic operator for which efficient solution
methods exist. For example, the preconditioner $P = A_F \Li$ could be
replaced by a small number of multigrid iterations, which would lead
to a solver with optimal complexity. Here, we will ignore these
additional complications and assume that $P$ is inverted exactly.

Throughout the present section, the iteration matrix is given by
\begin{equation}\label{itmatrix}
  \Gqcl := I-\alpha (\Lqcl)^{-1}\Lqcf=I - \alpha (A_F \Li)^{-1} \Lqcf,
\end{equation}
where $\alpha > 0$ and $A_F = \phi_F'' + 4 \phi_{2F}'' > 0$. We will
investigate whether, if $\Us$ is equipped with a suitable topology,
$\Gqcl$ becomes a contraction. To demonstrate that this is a
non-trivial question, we first show that in the spaces $\Us^{1,p}$, $1
\leq p < \infty$, which are natural choices for elliptic operators,
this result does not hold.

\begin{proposition}
  \label{th:contM_U12}
  If $2 \leq K \leq N/2,$ $\phi_{2F}'' \neq 0,$ and $p
  \in [1, \infty),$ then for any $\alpha>0$ we have
  \begin{displaymath}
    \left\| \Gqcl \right\|_{\Us^{1,p}} \sim N^{1/p} \qquad \text{as } N \to \infty.
  \end{displaymath}
\end{proposition}
\begin{proof}
We have from \eqref{equiva} and $q=p/(p-1)$ the inequality
\begin{align*}
    \big\| \Li^{-1} \Lqcf \big\|_{\Us^{1,p}}
    =~& \max_{\substack{u \in \Us\\\|u'\|_{\ell^p_\eps} = 1}}
    \big\| \big( \Li^{-1} \Lqcf u\big)' \big\|_{\ell^p_\eps} \\
    \le~& 2\max_{\substack{u, v \in \Us\\\|u'\|_{\ell^p_\eps} = 1,\ \|v'\|_{\ell^q_\eps} = 1}}
    \Big \< \big(\Li^{-1} \Lqcf u\big)', v' \Big\> \\
    =~& 2\max_{\substack{u, v \in \Us\\\|u'\|_{\ell^p_\eps} = 1,\ \|v'\|_{\ell^q_\eps} = 1}}
    \Big \< \Li \big(\Li^{-1} \Lqcf u\big), v\Big \> \\
    =~& 2\max_{\substack{u, v \in \Us\\\|u'\|_{\ell^p_\eps} = 1,\ \|v'\|_{\ell^q_\eps} = 1}}
    \big\< \Lqcf u, v \big\> \\
    =~&  2 \big\| \Lqcf \big\|_{L(\Us^{1,p},\  \Us^{-1,p})}
  \end{align*}
  as well as the reverse inequality
  \begin{displaymath}
  \big\| \Lqcf \big\|_{L(\Us^{1,p},\  \Us^{-1,p})}\le
  \big\| \Li^{-1} \Lqcf \big\|_{\Us^{1,p}}.
  \end{displaymath}
  The result now follows from the definition of $\Gqcl$ in ~\eqref{itmatrix}, Lemma~\ref{lpbound},
  and the fact that $\alpha>0$ and $A_F > 0.$
\end{proof}

We will return to an analysis of the QCL preconditioner in the space
$\Us^{1,2}$ in Section \ref{sec:P_QCL_U12}, but will first attempt to
prove convergence results in alternative norms.

\subsection{Analysis of the QCL preconditioner in $\Us^{2,\infty}$}
\label{sec:P_QCL_U2inf}
We have found in our previous analyses of the QCF method
\cite{dobs-qcf2,doblusort:qcf.stab} that it has superior properties in
the function spaces $\Us^{1,\infty}$ and $\Us^{2,\infty}$. Hence, we
will now investigate whether $\alpha$ can be chosen such that $\Gqcl$ is a
contraction, uniformly as $N \to \infty$. In \cite{doblusort:qcf.stab},
we have found that the analysis is easiest with the somewhat unusual
choice $\Us^{2,\infty}$. Hence we begin by analyzing $\Gqcl$ in this
space.

To begin, we formulate a lemma in which we compute the operator norm
of $\Gqcl$ explicitly. Its proof is slightly technical and is therefore
postponed to Appendix~\ref{app:proofs_norm}.

\begin{lemma}
  \label{th:normM_U2inf}
  If $N \geq 4$, then
  \begin{displaymath}
   \left \| \Gqcl \right\|_{\Us^{2,\infty}}
    = \Big|1 - \alpha\big( 1 - \smfrac{2\phi_{2F}''}{A_F}\big)\Big|
    + \alpha \Big|\smfrac{2\phi_{2F}''}{A_F}\Big|.
  \end{displaymath}
\end{lemma}

What is remarkable (though not necessarily surprising) about this
result is that the operator norm of $\Gqcl$ is independent of $N$ and
$K$. This immediately puts us into a position where we can obtain
contraction properties of the iteration matrix $\Gqcl$, that are uniform
in $N$ and $K$. It is worth noting, though, that the optimal
contraction rate is not uniform as $A_F$ approaches zero; that is, the
preconditioner does not give uniform efficiency as the system
approaches its stability limit.

\begin{theorem}
  \label{th:contM_U2inf}
  Suppose that $N \geq 4$, $A_F>0,$ and $\phi_{2F}'' \leq 0$, and define
  \begin{displaymath}
    \alphaoptqcl2 :=
    \frac{A_F}{A_F+2|\phi_{2F}''|}=\frac{2 A_F}{\phi_F''+A_F}
    \quad \text{and} \quad
    \alphamaxqcl2 :=
    \frac{2 A_F}{\phi_F''}.
  \end{displaymath}
  Then $\Gqcl$ is a contraction of $\Us^{2,\infty}$ if and only if $0
  < \alpha < \alphamaxqcl2$, and for any such choice the contraction
  rate is independent of $N$ and $K$. The optimal choice is $\alpha =
  \alphaoptqcl2$, which gives the contraction rate
  \begin{displaymath}
    \big\| G_{\rm qcl}\big(\alphaoptqcl2\big) \big\|_{\Us^{2,\infty}} =
    \smfrac{1 - {\frac{A_F}{\phi''_F}}}{1 + {\frac{A_F}{\phi''_F}}} < 1.
  \end{displaymath}
\end{theorem}
\begin{proof}
  Note that $\alphaoptqcl2 = 1 / \big(1 - \smfrac{2 \phi_{2F}''}{A_F}
  \big)$. Hence, if we assume, first, that $0 < \alpha \leq
  \alphaoptqcl2,$ then
  \begin{align*}
    \|\Gqcl\|_{\Us^{2,\infty}} = 1 - \alpha \big( 1 - 2 \smfrac{\phi_{2F}''}{A_F} \big)
    - 2 \alpha \smfrac{\phi_{2F}''}{A_F} = 1-\alpha =: m_1(\alpha).
  \end{align*}
  The optimal choice is clearly $\alpha = \alphaoptqcl2$ which gives the
  contraction rate
  \begin{displaymath}
    \big\|G_{\rm qcl}\big(\alphaoptqcl2\big)\big\|_{\Us^{2,\infty}}
    = \alphaoptqcl2 \Big| \frac{2\phi_{2F}''}{A_F}\Big|
    = \frac{2 |\phi_{2F}''|}{\phi_F'' + 2 \phi_{2F}''}
    = \smfrac{1 - {\frac{A_F}{\phi''_F}}}{
      1 + {\frac{A_F}{\phi''_F}}}.
  \end{displaymath}

  Alternatively, if $\alpha \geq \alphaoptqcl2,$ then
  \begin{displaymath}
    \big\|\Gqcl\big\|_{\Us^{2,\infty}} = \alpha \big(1 - \smfrac{4\phi_{2F}''}{A_F} \big) - 1
    = \alpha \frac{\phi_{F}''}{A_F} - 1 =: m_2(\alpha).
  \end{displaymath}
  This value is strictly increasing with $\alpha$, hence the optimal
  choice is again $\alpha = \alphaoptqcl2$.

  Moreover, we have $m_2(\alpha) < 1$ if and only if
  \begin{displaymath}
    \alpha < \frac{2 A_F}{\phi_F''} = \alphamaxqcl2.
  \end{displaymath}
  Since, for $\alpha = \alphaoptqcl2$ we have $m_1(\alpha) =
  m_2(\alpha) < 1,$ it follows that $\alphamaxqcl2 > \alphaoptqcl2$
  (as a matter of fact, the condition $\alphamaxqcl2 > \alphaoptqcl2$
  is equivalent to $A_F > 0$). In conclusion, we have shown that
  $\|\Gqcl\|_{\Us^{2,\infty}}$ is independent of $N$ and $K$ and that
  it is strictly less than one if and only if $\alpha <
  \alphamaxqcl2$, with optimal value $\alpha = \alphaoptqcl2$.
\end{proof}

As an immediate corollary, we obtain the following general convergence
result.

\begin{corollary}
  \label{th:cor_Xnorm}
  Suppose that $N \geq 4$, $A_F>0,$ $\phi_{2F}'' \leq 0$, and suppose
  that $\|\cdot\|_X$ is a norm defined on $\Us$ such that
  \begin{displaymath}
    \| u \|_X \leq C \| u \|_{\Us^{2,\infty}} \qquad \forall u \in \Us.
  \end{displaymath}
  Moreover, suppose that $0 < \alpha < \alphamaxqcl2$.  Then, for any
  $u \in \Us$,
  \begin{displaymath}
    \left\| \Gqcl^n u \right\|_X \leq \hat q^n C \|u\|_{\Us^{2,\infty}}
    \rightarrow 0 \quad \text{as } n \to \infty,
  \end{displaymath}
  where $\hat q := \|\Gqcl\|_{\Us^{2,\infty}} < 1$.

  In particular, the convergence is uniform among all $N$, $K$ and all
  possible initial values $u \in \Us$ for which a uniform bound on
  $\|u\|_{\Us^{2,\infty}}$ holds.
\end{corollary}
\begin{proof}
  We simply note that, according to Theorem \ref{th:contM_U2inf}, for
  $0 < \alpha < \alphamaxqcl2$, we have
  \begin{displaymath}
    \left\|\Gqcl^n \right\|_{\Us^{2,\infty}} \leq \hat q^n,
  \end{displaymath}
  where $\hat q := \|\Gqcl\|_{\Us^{2,\infty}} < 1$ is a number that is
  independent of $N$ and $K$. Hence, we have
  \begin{displaymath}
   \left \| \Gqcl^n u\right\|_X \leq C \left\|\Gqcl^n u \right\|_{\Us^{2,\infty}}
    \leq C \hat q^n \|u\|_{\Us^{2,\infty}}.
  \end{displaymath}
\end{proof}

\begin{remark}
  Although we have seen in Theorem \ref{th:contM_U2inf} and
  Corollary~\ref{th:cor_Xnorm} that the linear stationary method with
  preconditioner $A_F \Li$ and with sufficiently small step size
  $\alpha$ is convergent, this convergence may still be quite slow if
  the initial data is ``rough.''  Particularly in the context of
  defects, we may, for example, be interested in the convergence
  properties of this iteration when the initial residual is small or
  moderate in $\Us^{1,p}$, for some $p \in [1, \infty]$, but possibly
  of order $O(N)$ in the $\Us^{2,\infty}$-norm.  We can see from the
  following Poincar\'{e} and inverse inequalities
  \begin{displaymath}
  \|u\|_{\Us^{1,\infty}}\le \frac 12 \|u\|_{\Us^{2,\infty}}\qquad
  \text{and}\qquad \|u\|_{\Us^{2,\infty}}\le 2N \|u\|_{\Us^{1,\infty}} \qquad
  \text{for all } u\in\Us;
  \end{displaymath}
  that the application of Corollary~\ref{th:cor_Xnorm} to the case
  $X=\Us^{1,\infty}$ gives the estimate
  \begin{displaymath}
    \left\| \Gqcl^n u \right\|_{\Us^{1,\infty}} \leq  \hat q^n N \|u\|_{\Us^{1,\infty}} \qquad \text{for all } u\in\Us.
  \end{displaymath}
  Similarly, with $X = \Us^{1,2}$, we obtain
  \begin{equation}
    \label{eq:contrac_Gqcl_U12_1}
    \left\| \Gqcl^n u \right\|_{\Us^{1,2}} \leq  \hat q^n N^{3/2}
    \|u\|_{\Us^{1,2}} \qquad \text{for all } u\in\Us.
  \end{equation}

  We have seen in Proposition \ref{th:contM_U12} that a direct
  convergence analysis in $\Us^{1,p}$, $p < \infty$, may be difficult
  with analytical methods, hence we focus in the next section on the
  case $\Us^{1,\infty}$.
\end{remark}

\subsection{Analysis of the QCL preconditioner in $\Us^{1,\infty}$}
\label{sec:P_QCF_U1inf}
As before, we first compute the operator norm of the iteration matrix
explicitly. The proof of the following lemma is again postponed to the
Appendix~\ref{app:proofs_norm}.

\begin{lemma}
  \label{th:normM_U1inf}
  If $K \geq 3,\ N \geq \max(9, K+3)$, and $\phi_{2F}'' \leq
  0$, then
  \begin{equation*}
    \begin{split}
      \left\| \Gqcl \right \|_{\Us^{1,\infty}} =
      \begin{cases}
        \big|1-\alpha\big| + \alpha 4\big|\smfrac{\phi_{2F}''}{A_F}\big|
        & \hspace{-6.3mm} \text{for } 0 \leq \alpha \leq \alphaoptqclone, \\[2mm]
        \big|1 - \alpha \big( 1 - 2 \smfrac{\phi_{2F}''}{A_F} \big)\big|
        + \alpha (6 + 2 \eps - 4 \eps K) \big| \smfrac{\phi_{2F}''}{A_F} \big|
        & \text{for } \alphaoptqclone \leq \alpha,
      \end{cases}
    \end{split}
  \end{equation*}
  where
  \begin{displaymath}
    \alphaoptqclone := \Big[ 1 + (2 + \eps - 2 \eps K) \big| \smfrac{\phi_{2F}''}{A_F} \big| \Big]^{-1}
  \end{displaymath}
  satisfies $\alphaoptqcl2 \leq \alphaoptqclone \leq 1$.
\end{lemma}

Again we note that the operator norm is independent, but now up to
terms of order O($\eps K$), of the system size.

\begin{theorem}
  \label{th:contM_U1inf}
  Suppose that $K \geq 3,\ N \geq \max(9, K+3)$, and $\phi_{2F}'' < 0$,
  then the following statements are true:
  \begin{enumerate}
  \item[(i)] If $\phi_F'' + 8 \phi_{2F}'' \leq 0,$ then $\Gqcl$ is {\em
      not} a contraction of $\Us^{1,\infty}$, for any value of
    $\alpha$.
  \item[(ii)] If $\phi_F'' + 8 \phi_{2F}'' > 0,$ then $\Gqcl$ is a
    contraction for sufficiently small $\alpha$. More precisely,
    setting
    \begin{displaymath}
      \alphamaxqclone := \frac{2 A_F}{A_F + (8 + 2\eps-4\eps K)|\phi_{2F}''|},
    \end{displaymath}
    we have that $\Gqcl$ is a contraction of $\Us^{1,\infty}$ if and
    only if $0 < \alpha < \alphamaxqclone$. The operator norm
    $\|\Gqcl\|_{\Us^{1,\infty}}$ is minimized by choosing $\alpha =
    \alphaoptqclone$ (cf. Lemma \ref{th:normM_U1inf}) and in this case
    \begin{displaymath}
      \big\| G_{\rm qcl}\big(\alphaoptqclone\big) \big\|_{\Us^{1,\infty}} = 1 - \frac{\phi_F'' + 8 \phi_{2F}''}{
        \phi_F'' + (2 - \eps + 2 \eps K) \phi_{2F}''} < 1.
    \end{displaymath}
  \end{enumerate}
\end{theorem}
\begin{proof}
  Suppose, first, that $0 < \alpha \leq \alphaoptqclone$. Since
  $\alphaoptqclone \leq 1$ it follows that
  \begin{displaymath}
    \big\|\Gqcl\big\|_{\Us^{1,\infty}} = 1 - \alpha \frac{\phi_F'' + 8 \phi_{2F}''}{A_F},
  \end{displaymath}
  and hence $\|\Gqcl\|_{\Us^{1,\infty}} < 1$ if and only if $\phi_F'' + 8
  \phi_{2F}'' > 0$. In that case $\|\Gqcl\|_{\Us^{1,\infty}}$ is strictly
  decreasing in $(0, \alphaoptqclone]$.

  Since $\alphaoptqclone \geq \alphaoptqcl2 = (1 - 2
  \smfrac{\phi_{2F}''}{A_F})^{-1}$ we can see that
  $\|\Gqcl\|_{\Us^{1,\infty}}$ is always strictly increasing in
  $[\alphaoptqclone, +\infty)$ and hence if $\phi_F'' + 8 \phi_{2F}'' > 0$,
  then $\alpha = \alphaoptqclone$ minimizes the operator norm
  $\|\Gqcl\|_{\Us^{1,\infty}}$. Moreover, straightforward computations
  show that $\alphamaxqclone > \alphaoptqclone$ and that $\|\Gqcl\|_{\Us^{1,\infty}} <
  1$ if and only if $0 < \alpha < \alphamaxqclone$.
\end{proof}

We remark that the optimal value of $\alpha$ in $\Us^{1, \infty}$,
that is $\alpha = \alphaoptqclone$, is not the same as the optimal
value, $\alphaoptqcl2,$ in $\Us^{2,\infty}$. However, it is easy to
see that $\alphaoptqclone = \alphaoptqcl2 + O(\eps K)$, and hence,
even though $\alphaoptqcl2$ is not optimal in $\Us^{1,\infty}$ it is
still close to the optimal value.  On the other hand,
$\alphamaxqclone$ and $\alphamaxqcl2$ are not close, since, if $4\eps
K-2\eps<1,$ then
\begin{displaymath}
\alphamaxqclone \le
\frac{2A_F}{\phi_F''+3|\phi_{2F}''|}
< \frac{2A_F}{\phi_F''}
=\alphamaxqcl2.
\end{displaymath}

In summary, we have seen that the contraction property of $\Gqcl$ in
$\Us^{1,\infty}$ is significantly more complicated than in
$\Us^{2,\infty},$ and that, in fact, $\Gqcl$ is {\em not} a
contraction for all macroscopic strains $F$ up to the critical strain
$\Fcrit$.

\subsection{Analysis of the QCL preconditioner in $\Us^{1,2}$}
\label{sec:P_QCL_U12}
Even though we were able to prove uniform contraction properties for
the QCL-preconditioned iterative method in $\Us^{2,\infty}$, we have
argued above that these are not entirely satisfactory in the presence
of irregular solutions containing defects. Hence we analyzed the
iteration matrix $\Gqcl = I - \alpha(A_F\Li)^{-1} \Lqcf$ in
$\Us^{1,\infty}$, but there we showed that it is not a contraction up
to the critical load $\Fcrit$. To conclude our results for the QCL
preconditioner, we present a discussion of $\Gqcl$ in the space
$\Us^{1,2}.$

We begin by noting that it follows from \eqref{rec} that
\begin{equation*}
\begin{split}
  P^{1/2} e^{(n)} =~& P^{1/2} \Gqcl e^{(n-1)} = P^{1/2} \left(I - \alpha P^{-1} \Lqcf\right) P^{-1/2} \left(P^{1/2} e^{(n-1)}\right) \\
  =~& \left(I - \alpha P^{-1/2} \Lqcf P^{-1/2}\right) \left(P^{1/2} e^{(n-1)}\right) =: {\Gqclw} \left(P^{1/2} e^{(n-1)}\right).
\end{split}
\end{equation*}
Since $\| P^{1/2} v \|_{\ell^2_\eps} = A_F^{1/2}\| v \|_{\Us^{1,2}}$
for $v \in \Us$, it follows that $\Gqcl$ is a contraction in
$\Us^{1,2}$ if and only if ${\Gqclw}$ is a contraction in
$\ell^2_\eps$. Unfortunately, we have shown in Proposition
\ref{th:contM_U12} that $\|\Gqcl\|_{\Us^{1,2}} \sim N^{1/2}$ as $N \to
\infty$. Hence, we will follow the idea used in Section
\ref{sec:richardson} and try to find an alternative norm with respect
to which $ \Gqclw$ is a contraction.

From Lemma \ref{th:spec_U12} we deduce that there exists a similarity
transform $\tilS$ such that $\cond(\tilS) \leq N^2$, and such that
\begin{displaymath}
  \Li^{-1/2} \Lqcf \Li^{-1/2} = \tilS^{-1} \tilLam \tilS,
\end{displaymath}
where $\tilLam$ is the diagonal matrix of $\Us^{1,2}$-eigenvalues
$(\mu_j^\qnl)_{j = 1}^{2N-1}$ of $\Lqnl$. As an immediate consequence
we obtain
\begin{displaymath}
  \Gqclw = \tilS^{-1} \big(I - \smfrac{\alpha}{A_F} \tilLam\big) \tilS.
\end{displaymath}
Proceeding as in Section \ref{sec:richardson}, we would obtain that
$\| \Gqclw \|_{\ell^2_\eps} \leq O(N^2)$. Instead, we observe that
\begin{align*}
  \big\| \Gqcl u \big\|_{\tilS^T\tilS} =~&  \big\| \tilS \Gqclw u
  \big\|_{\ell^2_\eps}
  = \big\| (I - \smfrac{\alpha}{A_F} \tilLam) \tilS u
  \big\|_{\ell^2_\eps} \\
  \leq~& \big\| I - \smfrac{\alpha}{A_F} \tilLam \big\|_{\ell^2_\eps}
  \| \tilS u \|_{\ell^2_\eps}
  = \max_{j = 1, \dots,
    2N-1}\big| 1 - \smfrac{\alpha}{A_F} \mu_j^\qnl \big|
  \| u \|_{\tilS^T \tilS},
\end{align*}
that is,
\begin{equation}
  % \label{eq:Gqclw_est_1}
  \big\| \Gqclw \big\|_{\tilS^T\tilS} \leq \max_{j = 1, \dots,
    2N-1} \big| 1 - \smfrac{\alpha}{A_F} \mu_j^\qnl \big|.
\end{equation}
%
% Since the operators $\Li^{-1/2} \Lqcf \Li^{-1/2}$ and
% $\Li^{-1/2} \Lqnl \Li^{-1/2}$ have the same spectrum, there exists a
% similarity transformation $\tilde S$ such that
% \begin{displaymath}
%   \Li^{-1/2} \Lqcf \Li^{-1/2} = \tilde S^{-1} \Li^{-1/2} \Lqnl \Li^{-1/2} \tilde S.
% \end{displaymath}
% Repeating the argument in Section \ref{sec:richardson} verbatim, we
% obtain
% \begin{displaymath}
%   \| \Gqclw \|_{\tilde S^T\tilde S} = \max\big( \alpha A_F^{-1} \mu_{2N-1}^{\rm qnl} - 1, 1 - \alpha A_F^{-1} \mu^{\rm qnl}_{1} \big),
% \end{displaymath}
% where $\{ \mu_j^{\rm qnl}:j=1,\dots, 2N-1\}$ denotes the spectrum of
% $\Li^{-1/2} \Lqnl \Li^{-1/2}$.
Thus, we can conclude that ${\Gqclw}$ is a contraction in the
$\|\cdot\|_{\tilde S^T\tilde S}$-norm if and only if $0 < \alpha <
\alphamaxqclonetwo:=2 A_F / \mu_{2N-1}^{\rm qnl}$. Moreover, we obtain
the error bound
\begin{displaymath}
  \| e^{(n)} \|_{\Us^{1,2}} \leq
  \cond(\tilS)\, \tilde q^n \| e^{(0)} \|_{\Us^{1,2}}
  \leq N^2 \tilde{q^n} \| e^{(0)} \|_{\Us^{1,2}},
\end{displaymath}
where $\tilde q:=\big\| \Gqclw\big\|_{\tilde S^T \tilde S}$. This is
slightly worse in fact, than \eqref{eq:contrac_Gqcl_U12_1}, however,
we note that this large prefactor cannot be seen in the following
numerical experiment.

Moreover, optimizing the contraction rate with respect to $\alpha$
leads to the choice $\alphaoptqclonetwo:=2 A_F / (\mu_1^{\rm qnl} +
\mu_{2N-1}^{\rm qnl})$, and in this case we obtain from Lemma
\ref{th:spec_Lqnl_U12} that
\begin{displaymath}
 \tilde q = \tilde{q}_{\rm opt} :=
 \big\| {\widetilde G_{\rm qcl}\big(\alphaoptqclonetwo\big) }
 \big\|_{\tilde S^T \tilde S} = \frac{\mu_{2N-1}^{\rm qnl}
   - \mu_1^{\rm qnl}}{\mu_{2N-1}^{\rm qnl} + \mu_{1}^{\rm qnl}}
 \leq \frac{1-\frac{A_F}{\phi''_F}}  { 1+\frac{A_F}{\phi''_F}},
\end{displaymath}
where the upper bound is sharp in the limit $K \to \infty$.  It is
particularly interesting to note that the contraction rate obtained
here is precisely the same as the one in $\Us^{2,\infty}$ (cf. Theorem
\ref{th:contM_U2inf}). Moreover, it can be easily seen from Lemma
\ref{th:spec_Lqnl_U12} that $\alphaoptqclonetwo \rightarrow
\alphaoptqcl2$ as $K \to \infty$, which is the optimal stepsize
according to Theorem \ref{th:contM_U2inf}.  We further have that
$\alphamaxqclonetwo \rightarrow \alphamaxqcl2$ as $K \to \infty.$

\subsection{Numerical example for QCL-preconditioning}
We now apply the QCL-preconditioned stationary iterative method to the QCF
system with right-hand side \eqref{eq:num_ex_rhs},
$\phi_F''=1,$ $A_F=0.2$, and
the optimal value
$\alpha=\alphaoptqcl2$
(we note that $G_{\rm id}(\alphaoptqcl2)$ depends only on
$A_F/\phi_F''$ and $N,$ but $e^{(0)}$ depends on $A_F$ and
$\phi_F''$ independently).
The error for successive iterations in the
$\Us^{1,2}$, $\Us^{1,\infty}$ and $\Us^{2,\infty}$-norms are displayed
in Figure \ref{fig:stat_p}.  Even though our theory, in
this case, predicts a perfect contractive behavior only in
$\Us^{2,\infty}$ and (partially) in $\Us^{1,2}$, we nevertheless
observe perfect agreement with the optimal predicted rate also in the
$\Us^{1,\infty}$-norms. As a matter of fact, the parameters are chosen
so that case (i) of Theorem \ref{th:contM_U1inf} holds, that is,
$\Gqcl$ is {\em not} a contraction of $\Us^{1,\infty}$.  A possible
explanation why we still observe this perfect asymptotic behavior is
that the norm of $\Gqcl$ is attained in a subspace that is never
entered in this iterative process. This is also supported by the fact
that the exact solution is uniformly bounded in $\Us^{2,\infty}$ as
$N, K \to \infty$, which is a simple consequence of Proposition
\ref{th:sharp_stab_U2inf}.

\begin{figure}
  \begin{center}
    \includegraphics[width=9cm]{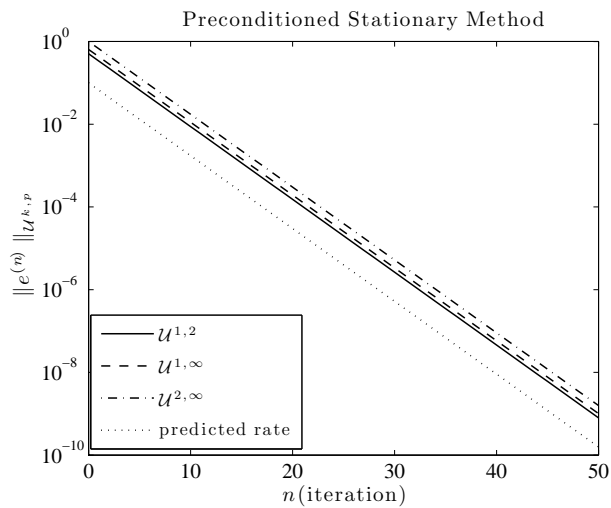}
  \end{center}
  \caption{\label{fig:stat_p} Error of the QCL-preconditioned linear
    stationary iterative method for the QCF system with $N = 800$, $K
    = 32$, $\phi_F''=1,$ $A_F=0.2$, optimal value $\alpha=\alphaoptqcl2,$ and right-hand side
    \eqref{eq:num_ex_rhs}. In this case, the iteration matrix $\Gqcl$ is
    {\em not} a contraction of $\Us^{1,\infty}$. Even though our
    theory predicts a perfect contractive behavior only in
    $\Us^{2,\infty}$, we observe perfect agreement with the optimal
    predicted rate also in the $\Us^{1,2}$ and
    $\Us^{1,\infty}$-norms.  }
\end{figure}

\section{Preconditioning with QCE ($P=\Lqce$): Ghost-Force Correction}
\label{gfc}
We have shown in \cite{Dobson:2008a,qcf.iterative} that the popular
{\em ghost force correction method (GFC)} is equivalent to
preconditioning the QCF equilibrium equations by the QCE equilibrium
equations.  The ghost force correction method in a quasi-static
loading can thus be reduced to the question whether the iteration
matrix
\begin{displaymath}
  \Gqce := I - (\Lqce)^{-1} \Lqcf
\end{displaymath}
is a contraction. Due to the typical usage of the preconditioner
$\Lqce$ in this case, we do not consider a step size $\alpha$ in this
section.  The purpose of the present section is (i) to investigate
whether there exist function spaces in which $\Gqce$ is a contraction;
and (ii) to identify the range of the macroscopic strains $F$ where
$\Gqce$ is a contraction.

We begin by recalling the fundamental stability result for the $\Lqce$
operator, Theorem \ref{th:stab_qce}:
\begin{displaymath}
  \inf_{\substack{u \in \Us \\ \|u'\|_{\ell^2_\eps} = 1}} \< \Lqce u, u\>
  = A_F + \lambda_K \phi_{2F}'',
\end{displaymath}
where $\lambda_K \sim \lambda_* + O(e^{-cK})$ with $\lambda_* \approx
0.6595$. This result shows that the GFC iteration must necessarily run
into instabilities before the deformation reaches the critical strain
$F_c^*$. This is made precise in the following corollary which states
that there is no norm with respect to which $\Gqce$ is a contraction up to
the critical strain $\Fcrit$.

\begin{corollary}
  \label{th:cor_failure_gfc}
  Fix $N$ and $K$, and let $\|\cdot\|_X$ be an arbitrary norm on the
  space $\Us$, then, upon understanding $\Gqce$ as dependent on $\phi_F''$
  and $\phi_{2F}''$, we have
  \begin{displaymath}
    \| \Gqce \|_X \to +\infty \quad
    \text{as} \quad
    A_F + \lambda_K \phi_{2F}'' \to 0.
  \end{displaymath}
\end{corollary}

Despite this negative result, we may still be interested in the
question of whether the GFC iteration is a contraction in ``very
stable regimes,'' that is, for macroscopic strains which are far away
from the critical strain $\Fcrit$. Naturally, we are particularly
interested in the behavior as $N \to \infty$, that is, we will
investigate in which function spaces the operator norm of $\Gqce$ remains
bounded away from one as $N \to \infty$. Theorem~\ref{th:nocoerc}
on the unboundedness of $\Lqcf$ immediately provides us with the
following negative answer.

\begin{proposition}
  \label{th:gfc:contG_H1}
  If $2 \leq K \leq N/2$, $\phi_{2F}'' \neq 0$, and $A_F + \lambda_K
  \phi_{2F}'' > 0$, then
  \begin{displaymath}
    \| \Gqce \|_{\Us^{1,2}} \sim N^{1/2} \qquad \text{as } N \to \infty.
  \end{displaymath}
\end{proposition}
\begin{proof}
  It is an easy exercise to show that, if $A_F + \lambda_K \phi_{2F}''
  > 0$, then the $\Us^{1,2}$-norm is equivalent to the norm induced by
  $\Lqce$, that is,
  \begin{displaymath}
    C^{-1} \|u\|_{\Us^{1,2}} \leq \|u\|_{\Lqce}
    \leq C \|u\|_{\Us^{1,2}}.
  \end{displaymath}
  Hence, we have $\| \Gqce \|_{\Us^{1,2}} \approx \| \Gqce \|_{\Lqce}$ and by
  the same argument as in the proof of Proposition \ref{th:contM_U12},
  and using again the uniform norm-equivalence, we can deduce that
  \begin{displaymath}
    \big\| \Gqce \big\|_{\Us^{1,2}} \approx\big \| \Lqcf \big\|_{L(\Us^{1,2},\  \Us^{-1,2})} \pm 1 \sim N^{1/2}
    \quad \text{as } N \to \infty.
  \end{displaymath}
\end{proof}

Since the operator $(\Lqce)^{-1} \Lqcf$ is more
complicated than that of $(A_F \Li)^{-1} \Lqcf$, which we analyzed in
the previous section, we continue to investigate the contraction
properties of $\Gqce$ in various different norms in numerical
experiments. In Figure \ref{fig:norm_Ggfc_1}, we plot the operator
norm of $\Gqce$, in the function spaces
\begin{displaymath}
  \Us^{k,p}, \quad k = 0, 1, 2, \quad p = 1, 2, \infty,
\end{displaymath}
against the system size $N$ (see Appendix~\ref{gnorm} for a
description of how we compute $\|\Gqce\|_{\Us^{k,p}}$). This
experiment is performed for $A_F / \phi_F'' =  0.8$ which is
at some distance from the singularity of $\Lqce$ (we note that
$\Gqce$ depends only on $A_F / \phi_F''$ and $N$ since
both $(\phi''_F)^{-1}\Lqcf$ and $(\phi_F'')^{-1}\Lqce$ depend
only on $A_F/\phi_F''$ and $N$).
The experiments
suggests clearly that $\|\Gqce\|_{\Us^{k,p}} \to \infty$ as $N \to
\infty$ for all norms except for $\Us^{1,\infty}$ and $\Us^{2,1}$.
\begin{figure}
  \begin{center}
    \includegraphics[width=10cm]{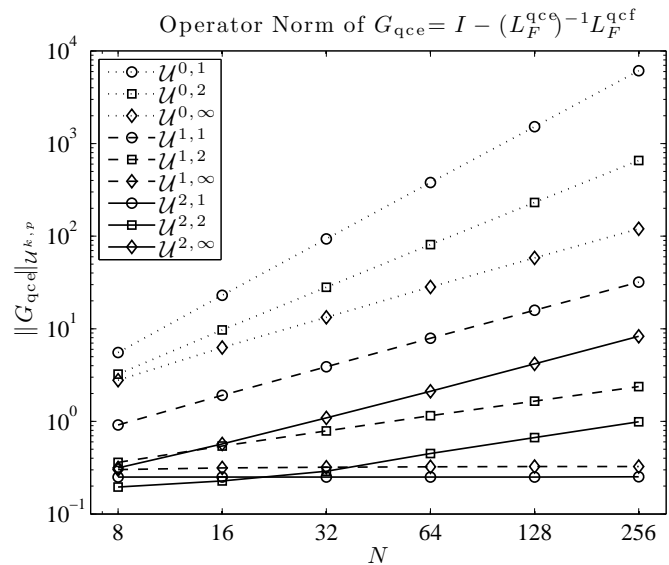}
  \end{center}
  \caption{\label{fig:norm_Ggfc_1} Graphs of the operator norm
    $\|\Gqce\|_{\Us^{k,p}}$, $k = 0,1,2$, $p = 1, 2, \infty$, plotted
    against the number of atoms, $N$, with atomistic region size $K =
    \lceil \sqrt{N} \rceil - 1$, and $A_F / \phi_F'' =  0.8$.
    (The graph for the $\Us^{1,p}$-norms, $p = 1,\infty$, are
    only estimates up to a factor of $1/2$; cf. Appendix \ref{gnorm}.)
    The graphs clearly indicate that $\|\Gqce\|_{\Us^{k,p}} \to
    \infty$ as $N \to \infty$ in all spaces except for
    $\Us^{1,\infty}$ and $\Us^{2,1}$.  }
\end{figure}

Hence, in a second experiment, we investigate how
$\|\Gqce\|_{\Us^{1,\infty}}$ and $\|\Gqce\|_{\Us^{2,1}}$ behave, for
fixed $N$ and $K$, as $A_F + \lambda_K \phi_{2F}''$ approaches
zero. The results of this experiment, which are are displayed in
Figure \ref{fig:norm_Ggfc_2}, confirm the prediction of
Corollary~\ref{th:cor_failure_gfc} that $\|\Gqce\|_{\Us^{k,p}} \to \infty$ as
$A_F + \lambda_K \phi_{2F}''$ approaches zero. Indeed, they show that
$\|\Gqce\|_{\Us^{k,p}} > 1$ already much earlier, namely around a
strain $F$ where $A_F \approx 0.52$ and $A_F + \lambda_K \phi_{2F}''
\approx 0.44$.

\begin{figure}[b]
  \begin{center}
    \includegraphics[width=8cm]{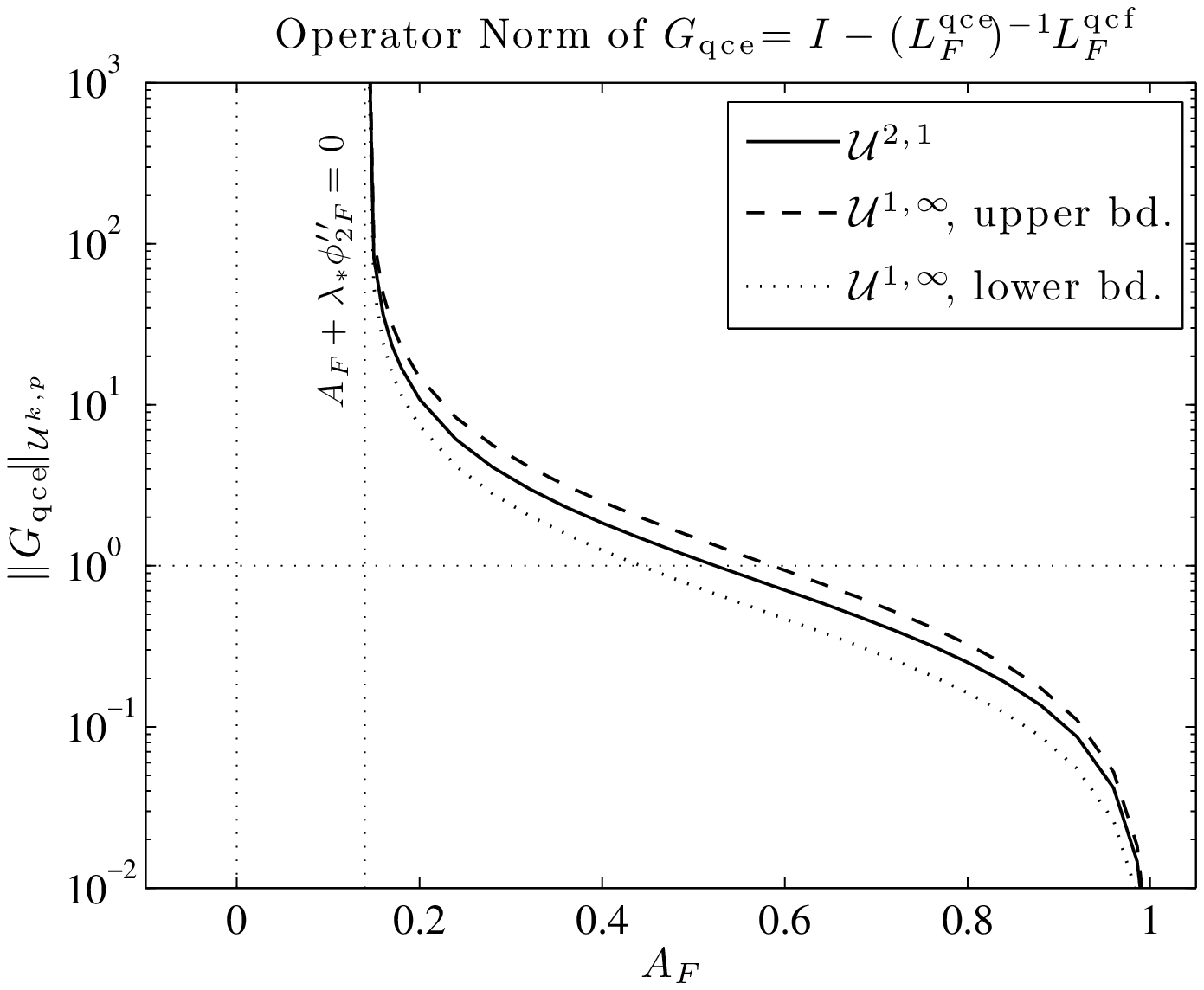}
  \end{center}
  \caption{\label{fig:norm_Ggfc_2} Graphs of the operator norm
    $\|\Gqce\|_{\Us^{k,p}}$, $(k, p) \in \{ (1,\infty), (2,1)\}$, for
    fixed $N = 256, K = 15$, $\phi_F'' = 1$, plotted against $A_F$.
    For the case $\Us^{1,\infty}$ only estimates are available and
    upper and lower bounds are shown instead (cf. Appendix
    \ref{gnorm}). The graphs confirm the result of Corollary
    \ref{th:cor_failure_gfc} that $\|\Gqce\|_{\Us^{k,p}} \to \infty$
    as $A_F + \lambda_K \phi_{2F}'' \to 0$. Moreover, they clearly
    indicate that $\|\Gqce\|_{\Us^{k,p}} > 1$ already for strains $F$
    in the region $A_F \approx 0.5$, which are much lower than the
    critical strain at which $\Lqce$ becomes singular. }
\end{figure}

Our conclusion based on these analytical results and numerical
experiments is that the GFC method is not universally reliable
near the limit strain $F_*,$ that
is, under conditions near the formation or movement of a defect
it can fail to converge to a stable solution of the QCF
equilibrium equations as the quasi-static loading step tends to zero
or the number of GFC iterations tends to infinity. Even though the
simple model problem that we investigated here cannot, of course,
provide a definite statement, it shows at the very least that further
investigations for more realistic model problems are required.

\section*{Conclusion}

We proposed and studied linear stationary iterative solution methods
for the QCF method with the goal of identifying iterative schemes that
are efficient and reliable for all applied loads. We showed that, if
the local QC operator is taken as the preconditioner, then the
iteration is guaranteed to converge to the solution of the QCF system,
up to the critical strain. What is interesting is that the choice of
function space plays a crucial role in the efficiency of the iterative
method. In $\Us^{2,\infty}$, the convergence is always uniform in $N$
and $K$, however, in $\Us^{1,\infty}$ this is only true if the
macroscopic strain is at some distance from the critical strain. This
indicates that, in the presence of defects (that is, non-smooth
solutions), the efficiency of a QCL-preconditioned method may be
reduced. Further investigations for more realistic model problems are
required to shed light on this issue.

We also showed that the popular GFC iteration must necessarily run
into instabilities before the deformation reaches the critical strain
$F_c^*$.  Even for macroscopic strains that are far lower than the
critical strain $\Fcrit,$ we show that $\| \Gqce \|_{\Us^{1,2}} \sim
N^{1/2}.$ We then give numerical experiments that suggest that
$\|\Gqce\|_{\Us^{k,p}} \to \infty$ as $N \to \infty$ for all tested
norms except for $\Us^{1,\infty}$ and $\Us^{2,1}.$

The results presented in this paper demonstrate the challenge for
the development of reliable and efficient iterative methods for
force-based approximation methods.  Further analysis and numerical experiments
for two and three dimensional problems are needed to more fully
assess the implications of the results in this paper for realistic
materials applications.

\newpage

\section{Appendix}

\subsection{Proof of Theorem \ref{th:stab_qce}}
\label{sec:app_gfc}
The purpose of this appendix is to prove the sharp stability result
for the operator $\Lqce$, formulated in Theorem \ref{th:stab_qce}.
Using Formula (23) in \cite{doblusort:qce.stab} we obtain the
following representation of $\Lqce$,
\begin{equation}
  \label{eq:Eqce_decomposition_final}
\begin{split}
  \big\<\Lqce u, u\big\> =~& \Bigg\{ \sum_{\ell = -N+1}^{-K-2} \eps A_F
  |u_\ell'|^2
  + \sum_{\ell = K+3}^N \eps A_F |u_\ell'|^2 \Bigg\} \\
%  \notag
  & + \Bigg\{\sum_{\ell = -K+2}^{K-1} \eps\Big(A_F |u_\ell'|^2 -
  \eps^2
  \phi_{2F}'' |u_\ell''|^2\Big)\Bigg\} \\
%  \notag
  & + \eps \Big\{ (A_F-\phi_{2F}'')(|u_{-K+1}'|^2 + |u_K'|^2) + A_F (|u_{-K}'|^2 + |u_{K+1}'|^2) \\
%  \notag
  & \hspace{2.5cm}  + (A_F+\phi_{2F}'') (|u_{-K-1}'|^2 + |u_{K+2}'|^2) \\
%  \notag
  & \qquad - \half \eps^2 \phi_{2F}'' (|u_{-K}''|^2 + |u_{-K-1}''|^2 +
  |u_{K}''|^2 + |u_{K+1}''|^2) \Big\}.
\end{split}
\end{equation}

If $\phi_{2F}'' < 0,$ then we can see from this decomposition that
there is a loss of stability at the interaction between atoms $-K-2$
and $-K-1$ as well as between atoms $K+1$ and $K+2$. It is therefore
natural to test this expression with a displacement $\hat u$ defined
by
\begin{displaymath}
  \hat u_\ell' = \begin{cases}1,& \ell = -K-1, \\
    -1, & \ell = K+2, \\
    0, & \text{otherwise}.
    \end{cases}
\end{displaymath}
From \eqref{eq:Eqce_decomposition_final}, we easily obtain
\begin{displaymath}
  \big\<\Lqce \hat u, \hat u \big\> = A_F + \half \phi_{2F}''.
\end{displaymath}
In particular, we see that, if $A_F + \half \phi_{2F}'' < 0,$ then
$\Lqce$ is indefinite. On the other hand, it was shown in
\cite{Dobson:2008b} that $\Lqce$ is positive definite provided $A_F +
\phi_{2F}'' > 0$. (As a matter of fact, the analysis in
\cite{Dobson:2008b} is for periodic boundary conditions, however,
since the Dirichlet displacement space is contained in the periodic
displacement space the result is also valid for the present case.)

Thus, we have shown that
\begin{displaymath}
  \inf_{\substack{u \in \Us \\ \|u'\|_{\ell^2_\eps} = 1}} \big\<\Lqce u, u \big\>
  = A_F + \mu \phi_{2F}'', \quad \text{where } \smfrac12 \leq \mu \leq 1.
\end{displaymath}
To conclude the proof of Theorem \ref{th:stab_qce}, we need to show that
$\mu$ depends only on $K$ and that the stated asymptotic result holds.

From \eqref{eq:Eqce_decomposition_final} it follows that $\Lqce$ can
be written in the form
\begin{displaymath}
 \big\<\Lqce u, u \big\> = (u')^T \mathcal{H} u',
\end{displaymath}
where we identify $u'$ with the vector $u' = (u_\ell')_{\ell =
  -N+1}^N$ and where $\mathcal{H} \in \R^{2N \times 2N}$. Writing
$\mathcal{H} = \phi_F'' \mathcal{H}_1 + \phi_{2F}'' \mathcal{H}_2,$ we
can see that $\mathcal{H}_1 = {\rm Id}$ and that $\mathcal{H}_2$ has
the entries
\begin{displaymath}
  \mathcal{H}_2 = {\footnotesize \left(
    \begin{array}{ccccccccccc}
      \ddots&\ddots&\ddots& & &  &   &    &      &      & \\
      & 1  &   2  &  1   &   &   &   &    &      &      & \\
      &    &   1  &  2   & 1 &   &   &    &      &      & \\
      &    &      &  1   &3/2&1/2&   &    &      &      & \\
      &    &      &      &1/2&3  &1/2&    &      &      & \\
      &    &      &      &   &1/2&9/2& 0  &      &      & \\
      &    &      &      &   &   &0  & 4  & 0    &      & \\
      &    &      &      &   &   &   & 0  & 4    & 0    & \\
      &    &      &      &   &   &   &    &\ddots&\ddots&\ddots
    \end{array}
  \right) }
\end{displaymath}
Here, the row with entries $[1,\, 3/2,\, 1/2]$ denotes the $K$th row (in
the coordinates $u_k'$). This form can be verified, for example, by
appealing to \eqref{eq:Eqce_decomposition_final}. Let $\sigma(A)$
denote the spectrum of a matrix $A$. Since, by assumption,
$\phi_{2F}'' \leq 0$, the smallest eigenvalue of $\mathcal{H}$ is
given by
\begin{displaymath}
  \min \sigma(\mathcal{H}) = \phi_F''
  + \phi_{2F}'' \max \sigma(\mathcal{H}_2),
\end{displaymath}
that is, we need to compute the largest eigenvalue $\bar\lambda$ of
$\mathcal{H}_2$. Since $\mathcal{H}_2 e_k = 4 e_k$ for $k = K+3,\,
K+4,\, \dots$ and for $K = -K-2, \,-K-3, \,\dots$, and since
eigenvectors are orthogonal, we conclude that all other eigenvectors
depend only on the submatrix describing the atomistic region and the
interface. In particular, $\bar\lambda$ depends only on $K$ but not on
$N$. This proves the claim of Theorem \ref{th:stab_qce} that $\lambda_K$
depends indeed only on $K$.

We thus consider the $\{-K-1, \dots, K+2\}$-submatrix
$\bar{\mathcal{H}}_2$, which has the form
\begin{displaymath}
  \bar{\mathcal{H}}_2 = {\footnotesize \left(
    \begin{array}{ccccccccc}
      9/2&1/2&   &   &   &   &   &   &    \\
      1/2& 3 &1/2&   &   &   &   &   &    \\
         &1/2&3/2& 1 &   &   &   &   &    \\
         &   & 1 & 2 & 1 &   &   &   &    \\
         &&&\ddots&\ddots&\ddots&&   &    \\
         &   &   &   & 1 & 2 & 1 &   &    \\
         &   &   &   &   & 1 &3/2&1/2&    \\
         &   &   &   &   &   &1/2& 3 &1/2 \\
         &   &   &   &   &   &   &1/2&9/2
    \end{array}
  \right). }
\end{displaymath}
Letting $\bar{\mathcal{H}}_2 \psi = \lambda \psi$, then for $\ell =
-K+2, \dots, K-1$,
\begin{displaymath}
  \psi_{\ell-1} + 2 \psi_\ell + \psi_{\ell+1} = \lambda \psi_\ell,
\end{displaymath}
and hence, $\psi$ has the general form
\begin{displaymath}
  \psi_\ell = a z^\ell + b z^{-\ell}, \qquad \ell = -K+1, \dots, K,
\end{displaymath}
leaving $\psi_\ell$ undefined for $\ell \in \{-K, -K-1, K+1, K+2\}$
for now, and where $z, 1/z$ are the two roots of the polynomial
\begin{displaymath}
  z^2 + (2-\lambda) z + 1 = 0.
\end{displaymath}
In particular, we have
\begin{equation}
  \label{eq:app:z_of_lambda}
  z = (\half \lambda - 1) + \sqrt{ (\half \lambda - 1)^2 - 1} > 1.
\end{equation}

To determine the remaining degrees of freedom, we could now insert
this general form into the eigenvalue equation and attempt to solve
the resulting problem. This leads to a complicated system which we
will try to simplify.

We first note that, for any eigenvector $\psi$, the vector $(\psi_{K -
  \ell})$ is also an eigenvector, and hence we can assume without loss
of generality that $\psi$ is skew-symmetric about $\ell = 1/2$. This
implies that $a = -b$. Since the scaling is irrelevant for the
eigenvalue problem, we therefore make the {\em ansatz} $\psi_\ell = z^\ell -
z^{-\ell}$. Next, we notice that for $K$ sufficiently large the term
$z^{-\ell}$ is exponentially small and therefore does not contribute
to the eigenvalue equation near the right interface. We may safely
ignore it if we are only interested in the asymptotics of the
eigenvalue $\bar\lambda$ as $K \to \infty$. Thus, letting
$\hat\psi_\ell = z^\ell$, $\ell = 1, \dots, K$ and $\hat\psi_\ell$
unknown, $\ell = K+1, K+2$, we obtain the system
\begin{align*}
  z^{K-1} + \smfrac32 z^K + \smfrac12 \hat\psi_{K+1} =~& \hat\lambda z^K, \\
  \smfrac12 z^K + 3 \hat\psi_{K+1} + \smfrac12 \hat\psi_{K+2} =~& \hat\lambda
  \hat\psi_{K+1}, \\
  \smfrac12 \hat\psi_{K+1} + \smfrac92 \hat\psi_{K+2} =~& \hat\lambda \hat\psi_{K+2}.
\end{align*}
The free parameters $\hat\psi_{K+1}, \hat\psi_{K+2}$ can be easily
determined from the first two equations. From the final equation we
can then compute $\hat\lambda$. Upon recalling from
\eqref{eq:app:z_of_lambda} that $\hat z$ can be expressed in terms of
$\hat\lambda$, and conversely that $\hat\lambda = (\hat z^2 + 1)/ \hat
z + 2$, we obtain a polynomial equation of degree five for $\hat z$,
\begin{displaymath}
q(\hat z) :=  4 \hat z^5 - 12 \hat z^4 + 9 \hat z^3 - 3 \hat z^2 - 4 \hat z + 2 = 0.
\end{displaymath}
Mathematica was unable to factorize $q$ symbolically, hence we
computed its roots numerically to twenty digits precision. It turns
out that $q$ has three real roots and two complex roots. The largest
real root is at $\hat z \approx 2.206272296$ which gives the value
\begin{displaymath}
  \hat\lambda = \frac{\hat z^2 + 1}{\hat z} + 2 \approx 4.659525505897.
\end{displaymath}
The relative errors that we had previously neglected are in fact of
order $\hat z^{-2K}$, and hence we obtain
\begin{equation*}
  \lambda_K = \lambda_* + O( e^{-cK} ), \qquad \text{where}
  \quad \lambda_* \approx 0.6595 \quad \text{and}
  \quad c \approx 1.5826.
\end{equation*}
This concludes the proof of Theorem \ref{th:stab_qce}.\qquad

\subsection{Proofs of Lemmas \ref{th:normM_U2inf} and \ref{th:normM_U1inf}}
\label{app:proofs_norm}
In this appendix, we prove two technical lemmas from Section
\ref{sec:P_QCL_U2inf}. Throughout, the iteration matrix $\Gqcl$ is given
by
\begin{displaymath}
  \Gqcl := I - \alpha (A_F \Li)^{-1} \Lqcf,
\end{displaymath}
where $\alpha > 0$ and $A_F = \phi_F'' + 4 \phi_{2F}'' > 0$. We begin
with the proof of Lemma \ref{th:normM_U2inf}, which is more
straightforward.

\begin{proof}[Proof of Lemma \ref{th:normM_U2inf}]
  Using the basic definition of the operator norm, and the fact that
  $\Li z = -z''$, we obtain
  \begin{equation*}
  \begin{split}
    \big\|\Gqcl\big\|_{\Us^{2,\infty}}
    = \max_{\substack{u \in \Us \\ \|u''\|_{\ell^\infty_\eps} = 1} }
    \big\| (\Gqcl u)'' \big\|_{\ell^\infty_\eps}
    = \max_{\substack{u \in \Us \\ \|u''\|_{\ell^\infty_\eps} = 1} }
    \big\| -\Li \Gqcl u \big\|_{\ell^\infty_\eps} .
    \end{split}
  \end{equation*}
  We write the operator $-\Li \Gqcl = -\Li + \smfrac{\alpha}{A_F} \Lqcf$ as
  follows:
  \begin{equation}
    \label{eq:app_co:D2_form_M}
    \big[-\Li \Gqcl u\big]_\ell =
    \begin{cases}
      u_\ell'' - \smfrac{\alpha}{A_F} \big( A_F u_\ell'' \big),
      & \text{if } \ell \in \Cs, \\[1mm]
      u_\ell'' - \smfrac{\alpha}{A_F} \big(\phi_F'' u_\ell'' +
      \phi_{2F}'' (u_{\ell-1}'' + 2 u_\ell'' + u_{\ell+1}'') \big),
      & \text{if } \ell \in \As.
    \end{cases}
  \end{equation}
  In the continuum region, we simply obtain
  \begin{displaymath}
    \big[-\Li \Gqcl u\big]_\ell = (1 - \alpha) u_\ell''
    \qquad \text{for } \ell \in \Cs.
  \end{displaymath}
  If $\ell \in \As$, we manipulate \eqref{eq:app_co:D2_form_M}, using
  the definition of $A_F = \phi_F'' + 4 \phi_{2F}''$, which yields
  \begin{align*}
    \big[-\Li \Gqcl u\big]_\ell
    =~& \Big[ 1 - \smfrac{\alpha}{A_F} \big( \phi_F'' + 2 \phi_{2F}''
    \big) \Big] u_\ell''
    + \Big[ - \smfrac{\alpha}{A_F} \phi_{2F}'' \Big] (u_{\ell-1}'' + u_{\ell+1}'' ) \\
    =~& \Big[ 1 - \alpha \big(1 - \smfrac{2\phi_{2F}''}{A_F}\big)\Big] u_\ell''
    + \Big[ - \alpha \smfrac{\phi_{2F}''}{A_F} \Big] (u_{\ell-1}'' + u_{\ell+1}'').
  \end{align*}
  In summary, we have obtained
  \begin{equation*}
    \big[-\Li \Gqcl u\big]_\ell =
    \begin{cases}
      [1-\alpha] u_\ell'',
      & \text{if } \ell \in \Cs, \\[1mm]
      \Big[ 1 - \alpha \big(1 - \smfrac{2\phi_{2F}''}{A_F} \big)\Big] u_\ell''
      + \Big[ - \alpha \smfrac{\phi_{2F}''}{A_F} \Big] (u_{\ell-1}'' + u_{\ell+1}'')
      & \text{if } \ell \in \As.
    \end{cases}
  \end{equation*}
  It is now easy to see that
  \begin{displaymath}
    \| \Gqcl \|_{L(\Us^{2,\infty},\ \Us^{2,\infty})}
    \leq \max\Big\{ \big|1 - \alpha\big|,
    \big|1 - \alpha \big( 1 - \smfrac{2\phi_{2F}''}{A_F}\big)\big|
    + \alpha \big|\smfrac{2\phi_{2F}''}{A_F}\big| \Big\}.
  \end{displaymath}
  As a matter of fact, in view of the estimate
  \begin{displaymath}
    \big|1 - \alpha \big( 1 - \smfrac{2\phi_{2F}''}{A_F}\big)\big|
    + \alpha \big|\smfrac{2\phi_{2F}''}{A_F}\big|
    \geq | 1 - \alpha| - \alpha \big|\smfrac{2\phi_{2F}''}{A_F}\big|
    + \alpha \big|\smfrac{2\phi_{2F}''}{A_F}\big|
    = |1-\alpha|,
  \end{displaymath}
  the upper bound can be reduced to
  \begin{equation}
    \label{eq:app_co:normM_upper_bound}
    \| \Gqcl \|_{L(\Us^{2,\infty},\  \Us^{2,\infty})}
    \leq \big|1 - \alpha \big( 1 - \smfrac{2\phi_{2F}''}{A_F}\big)\big|
    + \alpha \smfrac{2|\phi_{2F}''|}{A_F}.
  \end{equation}

  To show that the bound is attained, we construct a suitable test
  function. We define $u \in \Us$ via
  \begin{displaymath}
    u_{-1}''=u_{1}''={\rm sign}\Big[- \alpha \smfrac{2\phi_{2F}''}{A_F}\Big],
    \quad u_0'' = {\rm sign} \Big[ 1 - \alpha\big( 1 -  \smfrac{2\phi_{2F}''}{A_F}\big) \Big],
  \end{displaymath}
  (note that $0 \in \As$ for any $K \geq 0$) and the remaining values
  of $u_\ell''$ in such a way that $\sum_{\ell = -N+1}^N u_\ell'' =
  0$. If $N \geq 4,$ then there exists at least one function $u \in
  \Us$ with these properties and it attains the bound
  \eqref{eq:app_co:normM_upper_bound}. Thus, the bound in
  \eqref{eq:app_co:normM_upper_bound} is an equality, which concludes
  the proof of the lemma.
\end{proof}

Before we prove Lemma \ref{th:normM_U1inf}, we recall an explicit
representation of $\Li^{-1} \Lqcf$ that was useful in our analysis in
\cite{doblusort:qcf.stab}. The proof of the following result is completely analogous
to that of \cite[Lemma 14]{doblusort:qcf.stab} and is therefore sketched only
briefly. It is also convenient for the remainder of the section to
define the following atomistic and continuum regions for the strains:
\begin{displaymath}
  \As' = \big\{ -K+1, \dots, K \big\} \quad \text{and} \quad
  \Cs' = \big\{ -N+1, \dots, N \big\} \setminus \As'.
\end{displaymath}

\begin{lemma}
  \label{th:L1invLqcf_representation}
  Let $u \in \Us$ and $z = \Li^{-1} \Lqcf u$, then
  \begin{displaymath}
    z_\ell' = \sigma(u')_\ell - \overline{\sigma(u')}
    + \phi_{2F}'' \big( \tilde{\alpha}_{-K}(u') h_{-K, \ell}
    - \tilde{\alpha}_{K}(u') h_{K, \ell} \big),
  \end{displaymath}
  where $\sigma(u'),\ h_{\pm K} \in \R^{2N}$ and
  $\overline{\sigma(u')},\ \tilde{\alpha}_{\pm K}(u') \in \R$ are
  defined as follows:
  \begin{align*}
    \sigma(u')_\ell =~& \begin{cases}
      \phi_F'' u_\ell' + \phi_{2F}'' (u_{\ell-1}' + 2 u_\ell' + u_{\ell+1}'), &
      \quad \ell \in \As', \\
      (\phi_F'' + 4 \phi_{2F}'') u_\ell', & \quad \ell \in \Cs',
    \end{cases} \\
    \overline{\sigma(u')} =~& \frac{1}{2N} \sum_{\ell = -N+1}^N \sigma(u')_\ell
    = \smfrac{\eps}{2} \phi_{2F}'' \big[ u_{K+1}' - u_{K}' - u_{-K+1}' + u_{-K}' \big], \\
    \tilde{\alpha}_{-K}(u') =~& u_{-K+1}' - 2 u_{-K}' + u_{-K-1}', \quad
    \tilde{\alpha}_K(u') = u_{K+2}' - 2 u_{K+1}' + u_{K}', \quad \text{and} \\
    h_{\pm K, \ell} =~& \begin{cases}
      \smfrac12 (1 \mp \eps K), & \ell = -N+1, \dots, \pm K, \\
      \smfrac12 (-1 \mp \eps K), & \ell = \pm K+1, \dots, N.
    \end{cases}
  \end{align*}
\end{lemma}
\begin{proof}
  In the notation introduced above, the variational representation of
  $\Lqcf$ from \cite[Sec. 3]{doblusort:qcf.stab} reads
  \begin{displaymath}
    \< \Lqcf u, v \> = \< \sigma(u'), v' \>
    + \phi_{2F}'' \big[ \tilde{\alpha}_{-K}(u') v_{-K}
    - \tilde{\alpha}_{K}(u') v_K \big]
    \qquad \forall u, v \in \Us.
  \end{displaymath}
  Using the fact that $v_{\pm N} = 0$ and $\sum_\ell v_\ell' = 0$, it is
  easy to see that the discrete delta-functions appearing in this
  representation can be rewritten as
  \begin{displaymath}
    v_{\pm K} = \< h_{\pm K}, v' \>.
  \end{displaymath}
  Hence, we deduce that the function $z = \Li^{-1} \Lqcf$ is given by
  \begin{displaymath}
    \< z', v' \> = \< \Lqcf u, v\> = \big\< \sigma(u') + \phi_{2F}''[\tilde{\alpha}_{-K}(u') h_{-K} - \tilde{\alpha}_{K}(u') h_{K}], v' \>
    \quad \forall v \in \Us.
  \end{displaymath}
  In particular, it follows that
  \begin{displaymath}
    z' = \sigma(u') + \phi_{2F}''[\tilde{\alpha}_{-K}(u') h_{-K} - \tilde{\alpha}_{K}(u') h_{K}] + C,
  \end{displaymath}
  where $C$ is chosen so that $\sum_{\ell} z_\ell' = 0$. Since $h_{\pm
    K}$ are constructed so that $\sum_{\ell} h_{\pm K, \ell} = 0$, we
  only subtract the mean of $\sigma(u')$. Hence, $C = -
  \overline{\sigma(u')}$, for which the stated formula is quickly
  verified.
\end{proof}

\begin{proof}[Proof of Lemma \ref{th:normM_U1inf}]
  Let $u \in \Us$ with $\|u'\|_{\ell^\infty_\eps} \leq 1$.  Setting $z = \Gqcl
  u$, and employing Lemma \ref{th:L1invLqcf_representation}, we obtain
  \begin{align*}
    z_\ell' =~& u_\ell' - \smfrac{\alpha}{A_F} \Big[ \sigma_\ell(u') -
    \overline{\sigma(u')} + \phi_{2F}''(\tilde{\alpha}_{-K}(u') h_{-K,
      \ell} - \tilde{\alpha}_{K}(u') h_{K, \ell} )\Big] \\
    =~& \Big[ u_\ell' - \smfrac{\alpha}{A_F} \sigma_\ell(u') \Big]
    + \alpha \smfrac{\phi_{2F}''}{A_F} \Big[
    \smfrac{\eps}{2} (u_{K+1}' - u_{K}' - u_{-K+1}' + u_{-K}') \\
    & \hspace{5cm}
    - \tilde{\alpha}_{-K}(u') h_{-K, \ell}
    + \tilde{\alpha}_{K}(u') h_{K, \ell} \Big] \\
    :=~& R_\ell + S_\ell.
  \end{align*}
  We will estimate the terms $R_\ell$ and $S_\ell$ separately.

  To estimate the first term, we distinguish whether $\ell \in \Cs'$
  or $\ell \in \As'$. A quick computation shows that $R_\ell =
  (1-\alpha) u_\ell'$ for $\ell \in \Cs'$.  On the other hand, for
  $\ell \in \As'$ we have
  \begin{align*}
    R_\ell =~& \Big[ 1 - \smfrac{\alpha}{A_F} (\phi_F'' + 2 \phi_{2F}'') \Big] u_\ell'
    - \alpha\smfrac{\phi_{2F}''}{A_F} (u_{\ell-1}' + u_{\ell+1}') \\
    =~& \Big[ 1 - \alpha \big( 1 - \smfrac{2\phi_{2F}''}{A_F} \big) \Big] u_\ell' - \alpha \smfrac{\phi_{2F}''}{A_F} (u_{\ell-1}' + u_{\ell+1}')
    \qquad \forall \ell \in \As'.
  \end{align*}
  Since $\|u'\|_{\ell^\infty_\eps} \leq 1,$ we can thus obtain
  \begin{equation}
    \label{eq:upper_bnd_R}
    |R_\ell| \leq
    \begin{cases}
      |1-\alpha|, & \ell \in \Cs', \\
      \Big|1 - \alpha \big( 1 - \smfrac{2\phi_{2F}''}{A_F} \big)\Big| + \alpha \Big| \smfrac{2 \phi_{2F}''}{A_F} \Big|, & \ell \in \As'.
    \end{cases}
  \end{equation}
  As a matter of fact, these bounds can be attained for certain
  $\ell$, by choosing suitable test functions. For example, by
  choosing $u \in \Us$ with $u_N' = {\rm sign}(1-\alpha)$ we obtain
  $R_N = |1-\alpha|$, that is, $R_N$ attains the bound
  \eqref{eq:upper_bnd_R}. By choosing $u \in \Us$ such that
  \begin{displaymath}
    u_0' = u_2' = {\rm sign}\Big(-\smfrac{\phi_{2F}''}{A_F}\Big)=1
    \quad \text{and} \quad
    u_1' = {\rm sign}\Big( 1 - \alpha \big( 1 - \smfrac{2\phi_{2F}''}{A_F} \big) \Big),
  \end{displaymath}
  we obtain that $R_1$ attains the bound \eqref{eq:upper_bnd_R}.  In
  both cases one needs to choose the remaining free $u_\ell'$ so that
  $|u_\ell'| \leq 1$ and $\sum_{\ell} u_\ell' = 0$, which guarantees
  that such functions $u \in \Us$ really exist. This can be done under
  the conditions imposed on $N$ and $K$.

  To estimate $S_\ell,$ we note that this term depends only on a small
  number of strains around the interface. We can therefore expand it
  in terms of these strains and their coefficients and then maximize
  over all possible interface contributions. Thus, we rewrite $S_\ell$
  as follows:
  \begin{align*}
    S_\ell = \alpha \smfrac{\phi_{2F}''}{A_F} \Big\{
    ~& u_{-K-1}' [-h_{-K,\ell}]
    + u_{-K}' [2 h_{-K, \ell} + \smfrac{\eps}{2}]
    + u_{-K+1}' [-h_{-K,\ell}-\smfrac{\eps}{2}] \\
    & u_{K}' [h_{K,\ell} - \smfrac{\eps}{2}]
    + u_{K+1}' [-2h_{K,\ell} + \smfrac{\eps}{2}]
    + u_{K+2}' [h_{K,\ell}] \Big\}.
  \end{align*}
  This expression is maximized by taking $u_\ell'$ to be the sign of
  the respective coefficient (taking into account also the outer
  coefficient $\alpha \smfrac{\phi_{2F}''}{A_F}$), which yields
  \begin{align*}
    |S_\ell| \leq~& \alpha \big|\smfrac{\phi_{2F}''}{A_F} \big| \Big\{ |h_{-K,\ell}| + |2h_{-K,\ell} + \smfrac\eps2| + |h_{-K,\ell}
    + \smfrac\eps2| + |h_{K,\ell} - \smfrac\eps2| + |2h_{K,\ell} - \smfrac\eps2|
    + |h_{K,\ell}| \Big\} \\
    =~& \alpha \big|\smfrac{\phi_{2F}''}{A_F} \big| \Big\{ |4 h_{-K,\ell} + \eps| + |4 h_{K,\ell} - \eps| \Big\}.
  \end{align*}
  The equality of the first and second line holds because the terms
  $\pm \smfrac\eps2$ do not change the signs of the terms inside the
  bars. Inserting the values for $h_{\pm K, \ell},$ we obtain the bound
  \begin{equation*}
    |S_\ell| \leq
    \begin{cases}
      \alpha 4  \big|\smfrac{\phi_{2F}''}{A_F} \big|, & \ell \in \Cs', \\
      \alpha (4 + 2 \eps - 4 \eps K)  \big|\smfrac{\phi_{2F}''}{A_F} \big|, & \ell \in \As',
    \end{cases}
  \end{equation*}
  and we note that this bound is attained if the values for $u_\ell'$,
  $\ell = -K-1,-K, -K+1, K, K+1, K+2$, are chosen as described above.

  Combining the analyses of the terms $R_\ell$ and $S_\ell$, it
  follows that
  \begin{align*}
    \| z' \|_{\ell^\infty_\eps} \leq \max \Big\{&~ |1 - \alpha| + \alpha 4 \big|\smfrac{\phi_{2F}''}{A_F} \big|, \\
    &~ \big|1 - \alpha \big( 1 - \smfrac{2\phi_{2F}''}{A_F} \big)\big| + \alpha (6 + 2 \eps - 4 \eps K)  \big|\smfrac{\phi_{2F}''}{A_F}\big| \Big\}.
  \end{align*}
  To see that this bound is attained, we note that, under the
  condition that $K \geq 3$ and $N \geq K+3$, the constructions at the
  interface to maximize $S_\ell$ and the constructions to maximize
  $R_\ell$ do not interfere. Moreover, under the additional condition
  $N \geq \max(9, K+3)$, sufficiently many free strains $u_\ell'$
  remain to ensure that $\sum_{\ell} u_\ell' = 0$ for a test function
  $u \in \Us$, $\|u'\|_{\ell^\infty_\eps} = 1$, for which both
  $R_\ell$ and $S_\ell$ attain the stated bound. That is, we have
  shown that
  \begin{align*}
    \big\|\Gqcl\big\|_{\Us^{1,\infty}} = \max \Big\{& |1 - \alpha| + \alpha 4 \smfrac{|\phi_{2F}''|}{A_F}, \\
    & \big|1 - \alpha \big( 1 - \smfrac{2\phi_{2F}''}{A_F} \big)\big| + \alpha (6 + 2 \eps - 4 \eps K)  \smfrac{|\phi_{2F}''|}{A_F} \Big\} \\
    =: \max \{& m_\Cs(\alpha), m_\As(\alpha) \}.
  \end{align*}

  To conclude the proof, we need to evaluate this maximum explicitly.
  To this end we first define $\alpha_1 = ( 1 -
  \smfrac{2\phi_{2F}''}{A_F} \big)^{-1} < 1$. For $0 \leq \alpha \leq
  \alpha_1$, we have
  \begin{align*}
    m_\As(\alpha) =~& 1 - \alpha + \alpha(4 + 2\eps - 4\eps K) \big| \smfrac{\phi_{2F}''}{A_F}\big| \\
    \leq~&  1 - \alpha + \alpha 4 \big| \smfrac{\phi_{2F}''}{A_F}\big| = m_\Cs(\alpha),
  \end{align*}
  that is, $\|\Gqcl\|_{\Us^{1,\infty}} = m_\Cs(\alpha)$. Conversely, for
  $\alpha \geq 1$, we have
  \begin{align*}
    m_\As(\alpha) =~& \alpha \Big( 1 + (8+2\eps - 4 \eps K) \smfrac{|\phi_{2F}''|}{A_F} \Big) - 1 \\
    =~& m_\Cs(\alpha)+\alpha\Big( 4+2\eps - 4 \eps K) \smfrac{|\phi_{2F}''|}{A_F} \Big)
    \geq m_\Cs(\alpha),
  \end{align*}
  that is, $\|\Gqcl\|_{\Us^{1,\infty}} = m_\As(\alpha)$. Since, in
  $[\alpha_1, 1]$, $m_\Cs$ is strictly decreasing and $m_\As$ is
  strictly increasing, there exists a unique $\alpha_2 \in [\alpha_1,
  1]$ such that $m_\Cs(\alpha_2) = m_\As(\alpha_2)$ and such that the
  stated formula for $\|\Gqcl\|_{\Us^{1,\infty}}$ holds. A
  straightforward computation yields the value for $\alpha_2 =
  \alphaoptqclone$ stated in the lemma.
\end{proof}

\subsection{Computation of $\|\Gqce\|_{\Us^{k,p}}$}
\label{gnorm}
We have computed $\|\Gqce\|_{\Us^{k,p}}$ for $k=0, 2, p = 1, 2,
\infty$, from the standard formulas for the operator
norm~\cite{keller,saad} of the matrix $\Gqce$ and $\Li \Gqce \Li^{-1}$
with respect to $\ell_\eps^p$. For $k=1$ and $p = 2$, the norm is also
easy to obtain by solving a generalized eigenvalue problem.

The cases $k = 1$ and $p = 1, \infty$ are more difficult. In these
cases, the operator norm of $\Gqce$ in $\Us^{1,p}$ can be estimated in
terms of the $\ell_\eps^p$-operator norm of the conjugate operator
$\widehat G=I-(\Lqceconjf)^{-1}\Lqcfconjf: \R^{2N} \to \R^{2N}$ (see
Lemma~\ref{th:spec_Lqnl_U12} for an analogous definition of the
conjugate operator $\Lqnlconjf:\R^{2N}\to\R^{2N}$). It is not
difficult to see that $\|\Gqce\|_{\Us^{1,p}}=\|\widetilde
G\|_{\ell_\eps^p,\,\R^{2N}_*}$ for $\widetilde
G=I-(\Lqceconj)^{-1}\Lqcfconj:\R^{2N}_*\to\R^{2N}_*$ where we recall
that $\R^{2N}_* = \{ \varphi \in \R^{2N} : \sum_\ell \varphi_\ell = 0
\}$ (see Lemma~\ref{th:spec_Lqnl_U12} similarly for an analogous
definition of the restricted conjugate operator
$\Lqnlconj:\R^{2N}_*\to\R^{2N}_*$), it follows from \eqref{equiva}
that we have only computed $\|\Gqce\|_{\Us^{1,p}}$ for $p=1,\,\infty$
up to a factor of $1/2.$ More precisely,
\begin{displaymath}
  \|\Gqce\|_{\Us^{1,p}}
  \leq \|\widehat G\|_{\ell_\eps^p}
  \leq 2 \|\Gqce\|_{\Us^{1,p}}
\end{displaymath}

Finally We note that we can obtain $\Lqcfconjf$ from the
representation given in Lemma~\ref{th:L1invLqcf_representation} and
that $\Lqceconjf$ can be directly obtained from ~\eqref{qce}.

\end{document}